%% file: higherscars.tex
\documentclass[12pt,reqno]{amsart}
  
\usepackage{latexsym}
\usepackage{amssymb, tikz}

\usepackage{color}

\newcommand{\defeq}{\stackrel{{\rm def}}{=}}

\newcommand{\vareps}{\varepsilon}
\renewcommand{\epsilon}{\varepsilon}
\renewcommand{\phi}{\varphi}

\newcommand{\h}{\hbar}

\newcommand{\la}{\langle}
\newcommand{\ra}{\rangle}

\newcommand{\tN}{\tilde{N}}
\newcommand{\tM}{\tilde{M}}

\numberwithin{equation}{section}

\usepackage{amsthm}
\theoremstyle{plain}
\newtheorem*{theorem*}{Theorem}
\newtheorem{theo}[equation]{{\sc Theorem}}

\newtheorem{cor}[equation]{{\sc Corollary}}

\newtheorem{problem}[equation]{{\sc Problem}}
\newtheorem{conj}[equation]{{\sc Conjecture}}

\newtheorem{lem}[equation]{{\sc Lemma}}

\newtheorem{prop}[equation]{{\sc Proposition}}

\theoremstyle{definition}
\newtheorem{defn}[equation]{{\sc Definition}}

\theoremstyle{remark}
\newtheorem{rem}[equation]{Remark}
\theoremstyle{assumption}
\newtheorem{assumption}[equation]{Assumption}

\newcommand{\comments}[1]{}

\newcommand{\bequ}{\begin{equation}}
\newcommand{\eequ}{\end{equation}}

\include{mydefs}

\newcommand\DM{\Delta_M}
\newcommand\DN{\Delta_N}
\newcommand\DS{\Delta_{S^{r-1}}}

\newcommand\OpWh{\Oph^{\textrm{w}}}
\newcommand\Psih{\Psi_\h}
\newcommand\bPsih{\boldsymbol{\Psi}_\h}
\newcommand\psih{\psi_\h}

\newcommand\phih{\phi_\h}
\newcommand\Eh{E_\h}

\newcommand\mh{\mu_\h}

\newcommand\mb{\bar\mu}
\newcommand\mhb{\mb_\h}
\newcommand\musc{\mu_{\textrm{sc}}}

\newcommand\qq{\textrm{quad}}
\newcommand\nq{\textrm{nq}}

\title[Higher dimensional scars]{Scarring of quasimodes on hyperbolic manifolds}
\author{Suresh Eswarathasan}
\address{School of Mathematics, Cardiff University, Senghennyd Road, Cardiff, Wales, United Kingdom}
\email{eswarathasans@cardiff.ac.uk}
\author{Lior Silberman}
\address{University of British Columbia, Department of Mathematics, 1984 Mathematics Road, Vancouver, BC V6T 1Z2}
\email{lior@math.ubc.ca}

\begin{document}

\begin{abstract} 
Let $N$ be a compact hyperbolic manifold, $M\subset N$ an embedded totally
geodesic submanifold, and let $-\h^2\DN$ be the semiclassical
Laplace--Beltrami operator.  

For any $\vareps>0$ we explicitly construct families of \emph{quasimodes}
of spectral width at most $\vareps\frac{\h}{|\log\h|}$ which exhibit a
``strong scar'' on $M$ in that their microlocal lifts
converge weakly to a probability measure which places positive weight on $S^*M$ $( \embed S^*N)$.
An immediate corollary is that \emph{any} invariant measure on $S^*N$
occurs in the ergodic decomposition of the semiclassical limit of certain quasimodes of width $\vareps \frac{\h}{|\log\h|}$.
\end{abstract}

\maketitle

\section{Introduction}
We consider a problem in the spectral asymptotics of the
Laplace--Beltrami operator $\DN$ on a compact Riemannian manifold $N$.
Following the ``semiclassical'' convention we will index our eigenvalues and
(approximate) eigenfunctions by a spectral parameter $\h$ tending to zero
(the corresponding eigenvalue being $\lambda_\h = \h^{-2}$).  Abusing
notation we may have $\h$ tend to zero along a discrete sequence of values
without making this explicit.

As discussed in greater detail below, many results on the concentration
behaviour of exact eigenfunctions apply to certain approximate
eigenfunctions as well and we address here the converse problem of
constructing approximate eigenfunctions with prescribed concentration
behavior.  We start by specifying the relevant notion of
``approximate eigenfunction'', a relaxation of the eigenfunction equation
$\DN\Psih = \frac{\Eh}{\h^2} \Psih$:

\begin{defn} \label{d:logscale}
Fix $C>0$ and a sequence of \emph{energies} $\{\Eh\}_{\h}$ tending to $E_0>0$.  A family of \emph{quasimodes of width}
$\frac{C\h}{\abs{\log\h}}$ with \emph{central energies} $\Eh$ is
a sequence $\left\{\Psih\right\}_\h$ of $L^2(N)$-normalized
functions on $N$ such that
$$\norm{\left(-\h^2\DN-\Eh\right)\Psih}_{L^2(N)} \leq \frac{C\h}{\abs{\log\h}}.$$  
\end{defn}
If we prefer not to specify $C$ we will use the term ``\emph{log-scale} quasimodes''.  In the setting of our article, we will eventually set $\Eh = 1 + \mathcal{O}(h)$ for all $\h$.  Finally, note the existence of
such a non-zero quasimode as above shows that $-\h^2\DN$ has an eigenvalue $E$
in the interval
$\left[\Eh - \frac{C\h}{\abs{\log\h}}, \Eh + \frac{C\h}{\abs{\log\h}}\right]$.

We will study measure-theoretic concentration in the weak-* limit.
Precisely, to a quasimode $\Psih$ we associate the linear functional
$\mhb$ on $N$ given by $\mhb(f) = \int_N \abs{\Psih}^2 f \, dV$ for test
functions $f$ on $N$, where $dV$ is the Riemannian volume form on $N$.

These measures have natural lifts to distributions $\mh$ on
the cotangent bundle $T^*N$ (commonly referred to as ``microlocal lifts'');
we review the construction later.
A weak-* limit of these microlocal lifts is necessarily a probability measure
on $T^*N$. These limits, to be denoted $\musc$ and called
\emph{quantum limits} or \emph{semiclassical measures}, are the subject
of this paper.  Informally, they may be called weak-{*} limits of the
quasimodes themselves.

The reader is advised that these terms (quantum limit and semiclassical
measure) are usually reserved for the case where the $\Psih$ are
exact eigenfunctions, but the more general use is appropriate here. We will be clear in each invocation of these terms.

Our main result is the following

\begin{theorem*}[Cor.\ \ref{c:sigmund}]
Let $N$ be a compact hyperbolic manifold (that is, a compact Riemannian manifold
of constant negative curvature) then there exists $C>0$ such that any
probability measure $\mu$ on the unit cotangent bundle $SN$, which is invariant under
the geodesic flow, arises as the quantum limit of quasimodes of width $\frac{C \h}{|\log \h|}$.
\end{theorem*}

In the next part of the introduction we motivate our work by reviewing
the quantum unique ergodicity problem and various results towards it.  Knowledgeable
readers may wish to skip to Section~\ref{subsec:results} where we discuss
all our results.

\subsection{The quantum unique ergodicity problem}

The cotangent bundle $T^{*}N$ is naturally the phase space for a single
particle moving on our manifold $N$.  We fix a quantization scheme $\Oph$,
assigning to each \emph{observable}, which is a smooth function $a$ on $T^*N$
belonging to an appropriate symbol class, an operator
$\Oph(a)\colon L^{2}(N)\to L^{2}(N)$.
Fix a positive observable $H$, called the \emph{Hamiltonian},
and suppose that for a sequence of values of $\h$ tending to zero we
have chosen a corresponding sequence of normalized eigenfunctions
$\Psih\in L^{2}(N)$ where
\begin{equation}\label{eq:exact-ef}
\Oph(H)\Psih = \Eh\Psih.
\end{equation}
Let us suppose that $\Eh=E_0+\cO(\h)$ for some fixed $E_0>0$.  To each $\Psih$,
we associate its \emph{Wigner measure} which is the distribution $\mh$ on $T^{*}N$
given by
\[
\mh(a)=\left\langle \Psih,\Oph(a)\Psih\right\rangle \,.
\]

A major problem in spectral asymptotics is to
study the concentration behavior of the eigenfunctions $\Psih$ -- the
basic expectation is that the more chaotic the classical dynamics induced
by $H$, the more uniformly distributed the eigenfunctions $\Psih$ are
as $\h\to 0$.  One tool for studying this problem is the examination of subsequential
weak-* limits of the $\mu_\h$, and it is these objects that we call ``semiclassical measures'' or ``quantum limits''
$\musc$.  Their study was initiated in the works of Schnirel'man, Zelditch,
and Colin de Verdi\`ere 
\cite{Schnirelman:Avg_QUE,Zelditch:SL2_Lift_QE,CdV:Avg_QUE}, which we describe in more detail shortly.

Let us first review some aspects of quantization schemes.  First, since two different quantization schemes differ by terms of order
$\cO(\h)$, it follows that the measure $\musc$ depends on the sequence $\Psih$ but
not on the scheme itself.  Second, the existence of a positive
quantization scheme (the so-called Friedrich symmetrization) shows that
any limit must be a positive measure.  Using for $a$ the constant
function $1$ -- for which $\Oph(a)$ is the identity operator -- shows that
our limit $\musc$ is a probability measure.  Standard techniques also show
that the measure $\musc$ must be supported on the
\emph{energy surface }$\{H(x,\xi)=E_0\} \subset T^*N$.
Third, Egorov's Theorem relates the Hamiltonian flow induced by $H$ to
the action of the Schr\"odinger propagator
$U(t)=\exp\left(-\frac{it}{\h}\Oph(H)\right)$.
Since replacing $\Psih$ with its time-evolved state
$U(t)\Psih=\exp\left(-\frac{it\Eh}{\h}\right)\Psih$
has no effect on $\mh$ one can show that any limit $\musc$ must
be invariant under the Hamiltonian flow generated by $H$.
For a more general discussion of these properties, see the book
\cite{Zworski:SemiclassicalAnalysis} by Zworski.

\begin{problem}[Quantum Unique Ergodicity (``QUE")] \label{prob:que}
Classify, amongst the flow-invariant measures supported on level sets
of $H$, those which are weak-{*} limits of sequences of Wigner measures
of eigenfunctions.
\end{problem}

The case of a free particle moving on a compact
Riemmanian manifold where the classical dynamics is geodesic
flow corresponds to the Hamiltonian $H(x,\xi)=\norm{\xi}_{g}^{2}$;
here, we identify the
tangent and cotanget bundles using the metric $g$.  Its semiclassical
quantization $\Oph(H)$ is\footnote{Formally, this holds up to an
operator of order $\cO_{L^2}(h)$  -- perhaps it is better to use the
converse formulation that the principal symbol of $-h^2 \Delta_g$ is
$\left\vert \xi \right\vert^2_g$.} $-\h^{2}\Delta_{g}$,
where $\Delta_{g}$ is the Laplace--Beltrami operator for the metric $g$.
The eigenfunction equation \eqref{eq:exact-ef}
is then equivalent to the familar spectral problem $-\Delta\Psi=\lambda\Psi$ where we
take $\h=\sqrt{\frac{1}{\lambda}}$ and $E_{\h}=1$. For non-zero $E_0$, the energy surfaces $\{H = E_0\}$ are then all just dilations of 
the unit tangent bundle $SM$, which we call our \emph{energy surface}.  The metric on $N$ naturally gives rise to a geodesic-flow invariant measure, the Liouville measure $\mu_{L,N}$.

We briefly revert 
to indexing the eigenfunctions of $\DN$ in accordance with their set of non-decreasing eigenvalues so that we may state the following classical theorems:
\begin{theo}\label{thm:QE}
Let $N$ be a compact Riemannian manifold of dimension $n$, let
$\left\{\Psi_{n}\right\}_{n=0}^{\infty}\subset L^{2}(N)$
be an orthonormal basis of eigenfunctions of $\DN$ 
with corresponding eigenvalues $\left\{ \lambda_{n}\right\} _{n=0}^{\infty}$.
Write $\mu_{n}$ for the Wigner measure associated to the eigenfunction
$\Psi_{n}$.  We have that 
\begin{enumerate}
\item (Convergence on average/Generalized Weyl law)  The following statement on distributions holds:
\[
\frac{1}{T^{\frac{n}{2}}}\sum_{n<T}\mu_{n}\wklim{T\to\infty}\mu_{L,N}\,.
\]

\item (``Quantum ergodicity'') If the Liouville measure $\mu_{L,N}$ is
ergodic with respect to the geodesic flow, then there exists a subsequence
$\left\{ \Psi_{n_{k}}\right\} _{k=0}^{\infty}$ of density one along
which the Wigner measures converge to $\mu_{L}$.  Stated otherwise,
``almost all" eigenfunctions asymptotically equidistribute on $SN$.
\end{enumerate}
\end{theo}
A major case where the Liouville measure is ergodic is that of negatively
curved manifolds where we further have
\begin{conj}[Quantum Unique Ergodicity; Rudnick--Sarnak \cite{RudnickSarnak:Conj_QUE}]
For $N$ compact with negative sectional curvature, we have
$\mu_{n}\wklim{n\to\infty}\mu_{L,N}$.  In other words, the Liouville
measure is the \emph{unique} quantum limit.
\end{conj}

This conjecture predicts a form of uniform distribution of the eigenfunctions.  The opposite behavior, meaning an enhancement of some
eigenfunctions along closed geodesics, was observed numerically by Heller
\cite{Heller:Scarring} in the case of plane billiards and named
``scarring''.  Perhaps the strongest form of this phenomenon, a
``strong scar'' along an invariant measure, is the situation where a 
semiclassical measure has this singular measure as an atom in its ergodic
decomposition.

Some ergodic Euclidean billiards were shown to be quantum ergodic in the
sense of Theorem~\ref{thm:QE} by
G\'erard--Leichtnam \cite{GerardLeichtnam:QEboundary}
and Zelditch--Zworski \cite{ZelditchZworski:simpleQEboundary}.
The question of whether the numerically observed scarring persists in the
semiclassical limit for ergodic billiards was mostly settled by Hassell, who showed
in \cite{Hassell:NonUniqueBilliard} that for the one-parameter family
of billiards known as the Bunimovich stadium, almost every member has a sequence of eigenfunctions whose quantum limit puts positive
mass on the set of ``bouncing ball'' trajectories, a one-parameter family
of closed geodesics.  In negative curvature, closed orbits are
unstable so such families of trajectories cannot exist.  However, it is still an open question as to whether semiclassical measures on the stadium contain scars, particularly on the single closed orbit with endpoints on the circular wings.

Positive results toward the Rudnick--Sarnak Conjecture were
obtained by Lindenstrauss \cite{Lindenstrauss:SL2_QUE}
in the case of hyperbolic surfaces with the surface and the eigenfunctions
enjoying additional arithmetic symmetries; we note that these arguments were recently
simplified by Brooks--Lindenstrauss in \cite{BrooksLindenstrauss:OneHecke}.
To the knowledge of the authors, all positive results on general manifolds have depended on the breakthrough of
Anantharaman \cite{Anantharaman:QUE_Ent}.
She showed that for manifolds with geodesic flows having the Anosov
property, quantum limits must have \emph{positive entropy} with respect to
the action of the geodesic flow, which in particular
rules out the possibility of very singular measures being quantum
limits.  For example, a measure supported only a union of closed geodesics cannot occur as a semiclassical measure.
Follow-up works with generalizations and improvements include
\cite{AnantharamanNonnenmacher:HalfEntropyAnosov,AnantharamanKochNonnenmacher:BetterEntropy,AnantharamanSilberman:HaarComponent,Anantharaman:ICM2010Survey,Riviere:HalfEntropyNonpositiveCurv}.

The present article supports this line of work, specifically in studying to which extent these positive-entropy results are \emph{sharp}.

\subsection{Quasimodes}

Many of the positive-entropy results discussed in the previous section
continue to hold when one weakens the hypothesis that the $\bPsih$
are exact eigenfuctions.  In fact, the hypothesis that they are
\emph{log-scale quasimodes} in the sense of Definition \ref{d:logscale}
suffices.\footnote{The constant $C$ may vary from one sequence to another.}
For these we have that:
\begin{enumerate}
\item Every quantum limit of a sequence of log-scale quasimodes is a
probability measure supported on the energy surface $SN$.
\item This measure is invariant under the geodesic flow.
\item On a manifold of negative sectional curvature, any
weak-{*} limit of log-scale quasimodes with $C>0$ small enough has
positive entropy.\end{enumerate}

This version of Anantharaman's result is quantitative, in that the 
entropy bound depends on the spectral width parameter $C$.  Conversely
one may ask what is the smallest possible entropy of a quantum limit of
a sequence of quasimodes with such a width, or more generally
seek to classify those limits.

\begin{problem}[QUE for quasimodes]
Classify the weak-{*} limits of Wigner measures associated to log-scale
quasimodes.
\end{problem}

As we investigate this problem, we reserve the notation $\bPsih$ for
a sequence of quasidmoes for the rest of this paper.

The first result concerning limits of quasimodes was obtained by Brooks:
\begin{theo}[\cite{Brooks:QuasimodeScarring}]
Let $N$ be a compact hyperbolic surface and $\gamma\subset N$
be a closed geodesic.  Then for any $\epsilon>0$, there exists $\delta(\epsilon)>0$ and a sequence of quasimodes of width $\frac{\epsilon \h}{|\log \h|}$ and
central energies $\Eh$ whose semiclassical measure $\musc$ satisfies
$\mu_{sc}(\{ S\gamma\}) \geq \delta(\epsilon)$ where $S\gamma$ is the sphere bundle of $\gamma$. 
\end{theo}

Brooks's construction is analogous to one of Faure--Nonnenmacher--De Bi\`evre
\cite{FaureNonenmacherDeBievre:CatmapScar} in the toy model known
as the \emph{quantum cat map}.  In particular, the result depends on
the periodic boundary conditions on the hyperbolic surface and on
the connection between eigenfunctions of the Laplace operator and
the representation theory of $\SL_{2}(\R)$.

Using microlocal techniques instead, Nonnenmacher and the first named author
obtained a result for general Hamiltonians $H$ on a compact surface:
\begin{theo}[\cite{EswarathasanNonnenmacher:QuasimodeScarringSurfaces_preprint}]
\label{thm:EswarathasanNonnenmacher}
Let $\left(N,g\right)$ be a compact Riemannian surface and $\gamma$ be a \emph{hyperbolic} orbit of the Hamiltonian flow on an energy
surface $H^{-1}(E_0)$ for some regular energy $E_0>0$. Then for any $\epsilon>0$ there is a sequence
of quasimodes $\{\bPsih\}_{\h}$ of width at most $\epsilon\frac{\h}{\left|\log\h\right|}$ and central energies $E_{\h} = E_0 + \cO(\h / |\log \h|)$
of the quantum Hamiltonian $\Oph(H)$ whose semiclassical measure $\mu_{sc}$ has the property that $\mu_{sc}(\{S\gamma\}) \geq \frac{\epsilon}{\pi \lambda_{\gamma}} \frac{2}{3 \sqrt{3}} + \cO((\epsilon/\lambda_{\gamma})^2)$.  Here $\lambda_{\gamma}$ is the expansion rate along the unstable direction of the orbit $\gamma$.
\end{theo}

\subsection{Results of this paper}\label{subsec:results}

We seek to extend the above results to higher-dimensional manifolds
(but only for the semiclassical Laplace--Beltrami operator),
replacing the periodic geodesic $\gamma$ with a totally geodesic
submanifold $M$, and achieve this goal when $N$ is a hyperbolic
manifold.  However, our techniques are mostly microlocal and there is
reason to hope that they will apply in more general geometric settings.

The result is best thought of as a ``transfer" principle:
given a sequence of log-scale quasimodes on $M$ with associated quantum
limit $\mu_{sc}$, we extend them transverally using the hyperbolic dynamics
tranverse to $M$ to log-scale quasimodes on $N$ which still concentrate
on $SM$ in the same manner as the original sequence, or at least give it
positive mass.  To obtain specifically the Liouville measure $\mu_{L,M}$ on
$SM$, we use the equidistributed sequence provided by the Quantum Ergodicity
Theorem.  Our main result is the following:

\begin{theo} \label{t:main}
Let $N$ be a complete hyperbolic manifold, let $M\subset N$ be an embedded
compact totally geodesic submanifold, and pick a central energy $E_0>0$.
\begin{enumerate}
\item Suppose we are given a sequence of central energies energies
$\Eh = E_0 + \cO(\h)$ and some $\delta >0$.
Select a width constant which satisfies
$\pi \tilde{\lambda}_{E_0}(1 + \delta) \geq C := C_M \geq \pi \tilde{\lambda}_{E_0}$ where $\tilde{\lambda}_{E_0}= 2 \sqrt{E_0}$ is the expansion rate in the
unstable directions transverse to $M$.  Then there exists a sequence
of quasimodes $\{ \bPsih \}_{\h}$ of $N$ of width $\frac{C \h}{| \log \h|}$
for $\h \leq \h_0(M, \delta)$ and whose quantum limit is the
Liouville measure on $SM$.
\item Suppose in addition that $N$ is compact.  Then for any $\epsilon>0$
there exists a sequence of quasimodes with central energies converging to $E_0$ and of width $\frac{\epsilon\h}{\left|\log\h\right|}$ whose quantum limit
has $\mu_{L,M}$ as an ergodic component of mass at least
$\eta=\eta(C_M,\epsilon)>0$.
\item In both cases one may replace $\mu_{L,M}$ with any semiclassical measure $\mu_{sc}$ on $SM$ arising from log-scale quasimodes on $M$ but at the cost of widening the quasimodes to the
sum of the widths given in (1) (or (2)) and that of the quasimodes on $M$.
\end{enumerate}
\end{theo}

The bulk of the paper is devoted to establishing (1).  The reference
\cite[Sec.\ 6]{EswarathasanNonnenmacher:QuasimodeScarringSurfaces_preprint}
shows how to deduce (2) from (1) using spectral projection; see Lemma \ref{l:partloc} below.
As detailed in Lemma~\ref{l:ansatz-separates}, the same arguments automatically
establish (3) as well.  Unlike the previous results for surfaces, our methods are not restricted by the dimension of $M$ or by the codimension of $N$ in $M$.

We expect that the natural higher-dimensional generalization of
Theorem~\ref{thm:EswarathasanNonnenmacher} holds, with the only assumption
that the dynamics transverse to $M$ is hyperbolic.  We also find that
under the hypotheses of $N$ compact and negatively curved it suffices
to directly show concentration on closed geodesics
(exactly the totally geodesic $1$-dimensional submanifolds) in a strong
sense: once concentration on arbitrary closed geodesics is achieved
(with sufficient uniformity in the parameters), one can go beyond
log-scale semiclassical measures supported on totally geodesic
submanifolds and realize every invariant measure whatsoever as a log-scale
semiclassical measure:
\begin{cor} \label{c:sigmund}
Let $N$ be a compact hyperbolic manifold and let $\mu$ be a probability
measure on $SN$ invariant by the geodesic flow.  Then there is a sequence
of log-scale quasimodes whose associated semiclassical measure is $\mu$.
Furthermore, given $\epsilon>0$ there is $\eta=\eta(\epsilon)>0$ and a sequence
of quasimodes on $N$ of width $\frac{\epsilon\h}{\left|\log\h\right|}$ whose quantum limit $\mu$ carries weight at least $\eta(\epsilon)$ on the component $\mu$.\end{cor}
\begin{proof}
It was shown by Sigmund \cite{Sigmund:SpaceInvariantMeasures} that
there is a sequence $\left\{ \gamma_{k}\right\} _{k=1}^{\infty}$
of periodic geodesics whose natural measures $\delta_{k}=\delta_{S\gamma_{k}}$
converge in the weak-{*} topology to $\mu$. For each $k$ let $\left\{ \psi_{k,n}\right\} _{n=1}^{\infty}$
be a sequence of log-scale quasimodes on $N$ (all sequences of the same
spectral width) whose weak-{*} limit is $\delta_{S\gamma_{k}}$
guaranteed by Theorem \ref{t:main} (1).
By the separability of the space $C\left(SN\right)$ dual to the
space of measures, a diagonal argument -- using the uniformity in the
constants from Theorem \ref{t:main} (1) -- gives a subsequence
$\left\{ \psi_{k_{i},n_{i}}\right\} _{i=1}^{\infty}$
which is a sequence of log-scale quasimodes and which converges to the
desired limit $\mu$.
The spectral projection argument can again be used to shorten the spectral
width to $\frac{\epsilon \h}{\abs{\log \h}}$ while keeping at least
weight $\eta(\epsilon)$ on the component $\mu$.
\end{proof}

To the knowledge of the authors, this is the first result in the mathematical
quantum chaos literature which demonstrates that $o(\h)$ quasimodes
(which yield invariant semiclassical measures, log-scale or not) can
develop scars on fractal subsets.  As in the work of Brooks, our bounds
explore the extent to which the mentioned entropy bounds on quasimodes are sharp.

It is important to add that the construction of quasimodes which localize
along closed geodesics, or more generally along smooth invariant
submanifolds, has a long and rich history.  For a brief exposition of
this history, see the introduction of
\cite{EswarathasanNonnenmacher:QuasimodeScarringSurfaces_preprint}
and the references therein.

\subsection{Outline of the proof and further remarks}
The main idea of our proof is to combine the long-time evolution
idea which first originated in the work of
Vergini-Schneider \cite{VerginiSchneider:MatrixCoeffScars} and the
Fermi-normal coordinate/quantum Birkhoff normal idea of
Colin de Verdi\`ere-Parisse \cite{CdVParisse:UnstableScarsI}.

Let $N$ be a hyperbolic manifold and $M\subset N$ a compact totally
geodesic submanifold.  Then for $\epsilon_1>0$ small enough, the universal
cover of the $\epsilon_1$-neighbourhood $N_{\epsilon_1}(M)\subset N$
(thought of as a subset of the universal cover $\tilde{N}$ of $N$) is exactly the
$\epsilon_1$-neighbourhood $N_{\epsilon_1}(\tilde{M})$.  In particular, we have
a well-defined nearest-neighbor projection
$\pi\colon N_{\epsilon_1}(M)\to M$ and a coordinate system
$z\mapsto(\pi(z),x)$ identifying $N_{\epsilon_1}(M)$ with a product
$M\times B$ where $B$ is a Euclidean $\epsilon_1$-ball of dimension
$r = \dim N - \dim M$.  In this system the metric takes the form of a warped
product and we can separate the variables to get:
\begin{equation}
\DN = \frac{1}{1+|x|^2}\DM + \Delta_x +
      \left(x \cdot \frac{\partial}{\partial x}\right)^2 + 
      (n-1)\left(x \cdot \frac{\partial}{\partial x}\right)\,.
\end{equation}
Here $\DM$ is the Laplace--Beltrami operator the hyperbolic manifold $M$, $\Delta_x$ is the Euclidean Laplace operator in the tranverse variable $x$ and
$x \cdot \frac{\partial}{\partial x} = \sum_{i=1}^r x_i\frac{\partial}{\partial x_i}$
is the smooth radial differentiation operator written in Cartesian coordinates.  The calculation is reproduced in Section \ref{sub:Fermi-form} and crucially depends on the precise form and special symmetries of the hyperbolic metric.

For our desired log-scale quasimode we use the ansatz
\[
\bPsih(z)=\psih(x)\phih\left(\pi(z)\right)
\]
 where $\phih$ is an eigenfunction on $M$ of energy $E_0$, therefore making $\phih$
stationary for the Schr\"odinger propagator associated to $\DM$.  $\bPsih$
will concentrate on $M$ exactly when $\psih$ will concentrate on the
origin of the $x$-plane. This improves on the construction in
\cite{EswarathasanNonnenmacher:QuasimodeScarringSurfaces_preprint} where
the $\phih$ were particular localized wave packets on $M$ of width $\frac{C_{\gamma} \h}{| \log \h|}$ so that their
evolution (especially self-interference) was non-trivial and needed
to be dealt with using a particular normal form due to Sj\"ostrand \cite{Sjostrand:Resonances2D}; we will discuss this in greater detail shortly.  The use of exact eigenfunctions, or more generally
quasimodes of specified width, therefore simplifies the construction.

Plugging $\bPsih$ into the Schr\"odinger equation 
\[
-\h^2\Delta_{N}\bPsih = (E_0 + f(\h)) \bPsih
\]
gives that $\psih$ must be a quasimode of central energy $0$
for a second Schr\"odinger operator of the form
\[
-\h^{2}\Delta_x-\h^{2}\left(x\cdot \frac{\partial}{\partial x}\right)^2-(n-1)\h^{2}\left(x \cdot \frac{\partial}{\partial x}\right)-\frac{E|x|^{2}}{1+|x|^{2}} - f(\h) \,
\]
where $f(\h) = \mathcal{O}(\h)$.
The term $(n-1)\h^{2} \left( x \cdot \frac{\partial}{\partial x} \right)$ is lower order
in $\h$ and hence negligible in specific sense.  Ignoring for the moment the second order
term $\h^{2}\left(x\cdot \frac{\partial}{\partial x}\right)^2$ and using the approximation
$\frac{|x|^{2}}{1+|x|^{2}}\approx |x|^{2}$ (appropriate since our operator
will be applied to wave packets which are concentrated near $x=0$)
we have approximately the operator
\[
-\h^{2}\Delta_{x}-E_0\abs{x}^{2},
\]
commonly referred to as the \emph{inverted harmonic oscillator}.
This Hamiltonian is well-known to have log-scale quasimodes concentrating at
the origin, which is a hyperbolic fixed point of the associated classical
dynamics.

To realize this intuition rigorously we employ a \emph{Quantum
Birkhoff Normal Form} due to Iantchenko \cite{Iantchenko:BirhoffNormalForm98}.
This transformation, given by a semiclassical Fourier integral operator, approximates
our Schr\"odinger operator near its hyperbolic fixed point by an inverted
Harmonic oscillator together with controlled higher-order corrections.  This approach is not a novelty, but both the variation of $x \in \R^r$ rather
than $\R$ and our desire for uniformity sufficient for the argument of
Corollary~\ref{c:sigmund} require working through the arguments with some care.

\begin{rem}
While we sometimes use the fact that all Lyapunov exponents are equal,
the argument in the transverse direction via the use of the quantum Birkhoff normal form, should hold generally.

There is then reason to hope that our result could be extended to the
case where $N$ is a general manifold and $M\subset N$ is an invariant
hyperbolic submanifold: that is, a totally geodesic submanifold
where the dynamics transverse to $M$ are uniformly expanding and contracting.
The difficulty in establishing that case
is that the exact separation of variables available in constant negative
curvature would only hold in some approximate sense.  The transverse operator
would be more complicated and probably less explicit, and further work would
be required to establish the existance of a Birkhoff normal form.
\end{rem}

\begin{rem}
While we mananged to construct log-scale quasimodes concentrating
on general invariant measures, including those supported on fractals,
it would be interesting to construct them directly without appealing
to Sigmund's Theorem.  This would allow us to better explore the quantitative
relation between the spectral width of the quasimodes and the possible
entropies of the associated quantum limits: with the present method
even to obtain concentration on measures of relatively large entropy
we require the relatively large spectral width necessary to obtain
concentration on closed geodesics.  
\end{rem}

\begin{rem}
We note that in spite of our highly symmetric setting of constant negative curvature, it is non-trivial to obtain entropy lower bounds for measures associated to our sequence of quasimodes of width $\frac{\epsilon \h}{| \log \h|}$. While it is natural to believe that letting $\epsilon$ tend to 0 will give us a form of Anantharaman's lower bound, this would require transposing our construction into that of \cite{Anantharaman:QUE_Ent} and doing a precise analysis of various constants appearing in her estimates.  This is an ongoing project of the first named author and S. Nonnenmacher.
\end{rem}

\subsection*{Acknowledgements}
The authors would like to thank Alex Gorodnik, Joanna Karczmarek, Brian Marcus,
Jens Marklof and Roman Schubert for helpful discussions and the School
of Mathematics at the University of Bristol for their hospitality during
an important stage in the development of this paper.
The second named author would also like to thank Peter Sarnak for originally
suggesting this problem.
During the writing of this article, the first named author was supported by
a 10th Anniversary Grant from the Heilbronn Institute of Mathematical Research
and the second was supported by an NSERC Discovery Grant.


\section{Coordinates and Decompositions} \label{s:coords}

In this section we discuss the hyperbolic geometry from the symmetric space
point of view and provide an explicit Fermi normal coordinate system near
our selected totally geodesic submanifold $M$, which will allow us to
re-write the operator $h^2 \DN$ with respect to a warped-product structure.
This is the higher dimensional analogue of the idea in Section 6 of the
work of Colin de Verdi\`ere--Parisse \cite{CdVParisse:UnstableScarsI}.

\begin{rem}
Identifying the tangent and cotangent bundles using the Riemannian metric,
we may use the notations $TN$ and $T^*N$
(along with those for natural sub-bundles) interchangeably. 
\end{rem}
\subsection{A non-quantitative collar Lemma}
Our setup is as follows:
\begin{itemize}
\item Let $G$ be a semisimple Lie group with Iwasawa decomposition $G=NAK$,
$H<G$ be a semisimple closed subgroup containing $A$ such that $K_{H}=K\cap H$
is a maximal compact subgroup of $H$.
\item Let $S_{H}=H/K_{H}$ and $S=G/K$ be the corresponding symmetric spaces;
note that $S_{H}$ embeds in $S$ as a totally geodesic submanifold.
\item Let $\Gamma<G$ be a discrete subgroup such that $\Gamma_{H}=\Gamma\cap H$
is a uniform lattice in $H$.\end{itemize}

\begin{lem}
After passing to a finite-index subgroup we may assume $\Gamma_{H}\backslash\left(\Gamma\cap HK\right)=\left\{ 1\right\} $.\end{lem}
\begin{proof}
Let $\cF_{H}$ be a fundamental domain for $\Gamma_{H}\backslash H$.
Then the finite set $\Gamma\cap\cF_{H}K$ is a set of representatives
for the quotient. Since $\Gamma$ is residually finite there is a
subgroup $\Gamma^{1}$ not containing these elements (except for the
identity), so that such that $\Gamma^{1}\cap HK=\Gamma^{1}\cap H=\Gamma_{H}^{1}$.\end{proof}

\begin{assumption}
For the rest of this article, we let $M=\Gamma_{H}\backslash H/K_{H}$ be embedded in $N=\Gamma\backslash G/K$.
\end{assumption}

Now, let $\pi\colon S\to S_{H}$ be the projection onto a convex subset,
and write $N_{\epsilon}(S_{H})$ to be the $\epsilon-$neighborhood
of $S_{H}$ in $S$.

\begin{lem}
There exists an $\epsilon$ such that if $\gamma\in\Gamma$ has $\gamma N_{\epsilon}(S_{H})\cap N_{\epsilon}(S_{H})$ then $\gamma\in\Gamma_{H}$.
\end{lem}
\begin{proof}
If not there are $\gamma_{n}\in\Gamma$ not in $\Gamma_{H}$ and $z_{n}\in S$
such that $z_{n},\gamma_{n}z_{n}\in N_{\epsilon_{n}}(S_{H})$ with
$\epsilon_{n}\to0$. Let $y_{n}=\pi(z_{n})$, $y'_{n}=\pi(\gamma_{n}z_{n})$.
Let $\lambda_{n},\lambda'_{n}\in\Gamma_{H}$ be such that $\lambda_{n}y_{n},\lambda'_{n}\lambda y'_{n}\in\cF_{H}/K_{H}$.
Then replacing $z_{n}$ with $\lambda_{n}z_{n}$ and $\gamma_{n}$
with $\lambda'_{n}\lambda_{n}\gamma_{n}\lambda_{n}^{-1}$ we may assume
that $y_{n},y'_{n}\in\cF_{H}/K_{H}$ (and still $\gamma_{n}\notin\Gamma_{H}$).
Passing to a subsequence we may assume that $y_{n}\to y_{\infty}$,
$y'_{n}\to y'_{\infty}$. We also have $d(z_{n},y_{n})=d(\gamma_{n}z_{n},y'_{n})\leq\epsilon_{n}\to0$
so also $z_{n}\to y_{n}$, $\gamma_{n}z_{n}\to y'_{n}$. In particular
$\gamma_{n}$ move $y_{\infty}$ a bounded amount so it belongs to a finite
set, and we may further assume the sequence is a constant element $\gamma$
such that $\gamma y_{n}=y'_{n}$. But then $\gamma\in\Gamma\cap HK$
so $\gamma\in\Gamma_{H}$ as claimed.\end{proof}
\begin{cor}[Collar Neighborhood] \label{c:collar}
For $\epsilon$ small enough, $N_{\epsilon}\left(M\right)=\Gamma_{H}\backslash N_{\epsilon}\left(S_{H}\right)$.
\end{cor}

\subsection{Hyperbolic space}\label{sub:Fermi-form}

Suppose now $G=O(n,1)\supset O(m,1)=H$ and write $r=n-m$.
For $S$ take the upper half-space model with coordinates
$$\left(y, w_{1},\ldots,w_{m-1},w_{m},\ldots,w_{n-1}\right)
 = \left(y, w', w''\right)$$
in which $S_H = \left\{ \left(y, w', 0 \right)\right\}$.  Here $w' = (w_{1},\ldots,w_{m-1})$ and $w''= (w_{m},\ldots,w_{n-1})$.

Given a point $z=\left(y, w', w''\right)\in S$ let $\pi(z)\in S_{H}$
be the nearest-neighbor projection.  Writing $w'' \in\R^r$ in spherical
coordinates $\left(|w''|,\Omega \right)$
with $\Omega=\frac{1}{\abs{w''}}w''$,
the projection $\pi(z)$ is the point $\left(\eta,w'\right)$
where $\eta=\sqrt{y^{2}+|w''|^{2}}$. Writing $x=\frac{1}{y}w''$ with
magnitude $\abs{x} = \frac{|w''|}{y}$, the hyperbolic
distance between $z$ and $\pi(z)$ is 
\[
\rho=\rho\left(z,\pi(z)\right)=\arcsinh |x|=\log\left(|x|+\sqrt{1+|x|^{2}}\right)\,.
\]
We note that $|x|=\sinh\rho$, $\sqrt{1+|x|^{2}}=\cosh\rho$, $\eta = y\cosh\rho$.  We will eventually work with the coordinate system $\left(\eta,w',x\right)$ but as an intermediate step
also use the system $\left(\eta,w',\rho,\Omega \right)$ for which the inverse
map to the initial coordinate system is
$y=\frac{\eta}{\cosh\rho}$, $|w''|=\eta\tanh\rho$, $w''=(\eta\tanh\rho) \Omega$.

We can now compute the Laplace operator in these coordinates, starting
from the expression in the upper-halfspace coordinates $(y, w',w'')$
$$\DN = \left(y\frac{\partial}{\partial y}\right)^{2}
       -(n-1)\left(y\frac{\partial}{\partial y}\right)
       +y^{2}\left(\Delta_{w'}+ \Delta_{w''}\right)\,,$$
where $\Delta_{w'}=\sum_{i}\frac{\partial^{2}}{\partial (w'_{i})^{2}}$
and similarly for $w''$.
We note for reference that $\frac{d\rho}{d|x|}=\frac{1}{\sqrt{1+|x|^{2}}}$
so that $\frac{\partial\rho}{\partial y}=\frac{1}{\sqrt{1+|x|^{2}}}\frac{\partial |x|}{\partial y}$
and similarly for $w''_i$ instead of $y$.

By the collar neighborhood lemma, we may use the Fermi normal coordinate system (established
above for $S$ globally) locally in $N_\epsilon(M)$, therefore giving us:
\begin{lem} \label{l:simp-op-vol}
Using the coordinate system $\left(\eta,w', x\right)$ in the collar
neighborhood described in \eqref{c:collar},
the Laplace--Beltrami operator has the following Fermi-normal coordinate-type
structure: 
$$\DN=\frac{1}{1+|x|^2}\DM+\Delta_x + \left(\sum_{i=1}^r x_i\frac{\partial}{\partial x_i}\right)^2 + (n-1)\left(\sum_{i=1}^r x_i \frac{\partial}{\partial x_i}\right)$$
where $x \in \R^r$.  Furthermore, the volume element on $N$ in these
coordinates is
$$ \left( \eta^{-m} \, d \eta \, d^{m-1} w' \right) \left( (1+|x|^2)^{\frac{m-1}{2}} \, d^r x \right),$$
where the first term is the volume element of the hyperbolic metric on $M$.

\end{lem}
\begin{proof}  We give only the main parts of the calculation and leave the intermediate steps to the reader.

We have 
\begin{align*}
\frac{\partial}{\partial y} =\frac{1}{\cosh\rho}\frac{\partial}{\partial\eta}-\frac{1}{y}\tanh\rho\frac{\partial}{\partial\rho}\,.
\end{align*}
Using $y=\frac{\eta}{\cosh\rho}$ gives
\begin{align*}
y\frac{\partial}{\partial y} =\frac{1}{\cosh^{2}\rho}H-\left(\tanh\rho\right)R
\end{align*}
where we have set $H=\eta\frac{\partial}{\partial\eta}$ and $R=\frac{\partial}{\partial\rho}$.
Similarly we have
\begin{align*}
\frac{\partial}{\partial |w''|} =\frac{1}{\eta}\left((\tanh\rho)H+R\right)\,.
\end{align*}
Now, 
\begin{align*}
\left(y\frac{\partial}{\partial y}\right)^{2} =\frac{1}{\cosh^{4}\rho}H^{2}+\frac{2\sinh^{2}\rho}{\cosh^{4}\rho}H-\frac{2\sinh\rho}{\cosh^{3}\rho}HR+\frac{\sinh^{2}\rho}{\cosh^{2}\rho}R^{2}+\frac{\sinh\rho}{\cosh^{3}\rho}R\,
\end{align*}
 and 
\begin{align*}
\left(\frac{\partial}{\partial |w''|}\right)^{2}  =\frac{1}{\eta^{2}}\left(\frac{\sinh^{2}\rho}{\cosh^{2}\rho}H^{2}+\frac{1-\sinh^{2}\rho}{\cosh^{2}\rho}H+\frac{2\sinh\rho}{\cosh\rho}HR+R^{2}-\frac{\sinh\rho}{\cosh\rho}R\right)\,.
\end{align*}
Finally, the Euclidean Laplace operator in polar coordinates reads 
$$\Delta_{w''}=\frac{\partial^{2}}{\left(\partial |w''|\right)^{2}}+\frac{r-1}{|w''|}\frac{\partial}{\partial |w''|}+\frac{1}{|w''|^{2}}\DS.$$ 

A tedious calculation then gives for the coordinate system
$\left(\eta,w';\rho,\Omega \right)$
\begin{align*}
\DN =\frac{1}{\cosh^{2}\rho}\DM+ F +\frac{1}{\sinh^{2}\rho}\DS\,,
\end{align*}
where
$$\DM =  \left(\eta\frac{\partial}{\partial\eta}\right)^2
       - (m-1)\left(\eta\frac{\partial}{\partial\eta}\right)
       + \eta^2 \Delta_{w'}$$
is the Laplace--Beltrami operator of $S_H$ and
$$F = R^2 + \left((n-1)\tanh\rho+\frac{r-1}{\sinh\rho\cosh\rho}\right)R\,.$$

We now replace $\left(\rho,\Omega \right)$ with $\left(|x|,\Omega \right)$.  Using
$\frac{\partial}{\partial \rho} = \sqrt{1+|x|^2} \frac{\partial}{\partial |x|}$
leads to
\begin{align*}
F + \frac{1}{\sinh^2 \rho} E 
 & = \Delta_x + \left(\sum_{i=1}^r x_i\frac{\partial}{\partial x_i}\right)^2 + (n-1)\left(\sum_{i=1}^r x_i \frac{\partial}{\partial x_i}\right) \\
\end{align*}
where $\Delta_x$ is the Euclidean Laplacian.  Here, we have used that $|x| \frac{\partial}{\partial |x|} = \sum_{i=1}^r x_i\frac{\partial}{\partial x_i}$ after changing from polar coordinates back to Cartesian.
\end{proof}

\section{Semiclassical preliminaries} \label{s:hprelim}
In this sections we recall the concepts and definitions from
semiclassical analysis required for the sequel.  Notations are drawn from
the monographs \cite{Zworski:SemiclassicalAnalysis}. 

\subsection{Pseudodifferential operators on a manifold}
Recall the standard classes of symbols on $\mathbb{R}^{2d}$
\begin{eqnarray} \label{symbolclass}
S^{m}(\mathbb{R}^{2d}) &\defeq& \{a \in \Ci(\mathbb{R}^{2d} \times (0, 1] ): |\partial^{\alpha}_x\partial^{\beta}_{\xi} a| \leq C_{\alpha,\beta} \langle\xi\rangle^{m-|\beta|} \}.
\end{eqnarray}
Symbols in this class can be quantized through the $\h$-Weyl quantization into
the following pseudodifferential operators acting on $u\in \cS(\R^d)$:
\begin{equation}
\OpWh(a)\,u(x)\defeq
\frac{1}{(2\pi\h)^d}\int_{\mathbb{R}^{2d}}e^{\frac{i}{\h}\langle x-y,\xi\rangle}\,a\big(\frac{x+y}{2},\xi;\h\big)\,u(y)dyd\xi\,.
\end{equation}
One can adapt this quantization procedure to the case of the phase space
$T^*M$, where $M$ is a smooth compact manifold of dimension $d$
(without boundary) by performing the Fourier analysis in local coordinates.

In more detail, cover $M$ with a finite set $(f_l,V_l)_{l=1,\ldots,L}$ 
of coordinate charts, where each $f_l$ is a smooth diffeomorphism from
$V_l\subset M$ to a bounded open set $W_l\subset\R^{d}$.
To each $f_l$ correspond a pullback
$f_l^*:\Ci(W_l)\rightarrow \Ci(V_l)$ and a symplectic
diffeomorphism $\tilde{f}_l$ from $T^*V_l$ to $T^*W_l$: 
$$
\tilde{f}_l:(x,\xi)\mapsto\left(f_l(x),(Df_l(x)^{-1})^T\xi\right).
$$
Choose a smooth partition of unity $(\phi_l)_l$ adapted to our cover.
That means $\sum_l\phi_l=1$ and $\phi_l\in \Ci(V_l)$.
Then, any observable $a$ in $\Ci(T^*M)$ can be decomposed as
$a=\sum_l a_l$, where $a_l=a\phi_l$. Each $a_l$ belongs to
$\Ci(T^*V_l)$ and can be identified with the function
$\tilde{a}_l=(\tilde{f}_l^{-1})^*a_l\in \Ci(T^*W_l)$.
We now define the class of symbols of order $m$ on $T^*M$
(slightly abusing notation by treating $(x, \xi)$ as coordinates on $T^*W_l$)
via
\begin{eqnarray}
\label{pdodef}  S^{m}(T^*M)  &\defeq& \{a \in \Ci(T^*M \times (0, 1] ): a=\sum_l a_{l},\ \  \text{ such that } \\ 
\nonumber && \tilde{a}_l\in S^m(\R^{2d})\quad\text{for each }l\}.
\end{eqnarray}
This class is independent of the choice of cover and partition of unity. 
For any $a\in S^{m}(T^{*}M)$, one can associate to each component
$\tilde{a}_l\in S^{m}(\mathbb{R}^{2d})$ its Weyl quantization
$\OpWh(\tilde{a}_l)$, which acts on functions on $\R^{2d}$.
To get back to operators acting on $M$, we choose smooth cutoffs
$\psi_l\in \Cic(V_l)$ such that $\psi_l=1$ close to the support of
$\phi_l$, and set
\begin{equation}
\label{pdomanifold}\Oph(a)u \defeq 
\sum_l \psi_l\times\left(f_l^*\OpWh(\tilde{a}_l)(f_l^{-1})^*\right)\left(\psi_l\times u\right),\quad u\in \Ci(M)\,.
\end{equation}
This quantization procedure maps (modulo smoothing operators with seminorms $\mathcal{O}(\hbar^{\infty})$) symbols $a\in S^{m}(T^{*}M)$ onto the space $\Psi^{m}_\hbar(M)$ of semiclassical pseudodifferential 
operators of order $m$. The dependence in the cutoffs $\phi_l$ and $\psi_l$ only appears at order 
$\hbar\Psi^{m-1}_\hbar$ (Theorem 18.1.17 \cite{Hormander:PD_Ops_III} or Theorem 9.10 \cite{Zworski:SemiclassicalAnalysis}), so that the principal symbol map $\sigma_0:\Psi^{m}_\hbar(M)\rightarrow S^{m}(T^*M)/\h S^{m-1}(T^*M)$ is 
intrinsically defined. Most of microlocal calculus on $\R^d$
(for example the composition of operators, the Egorov and
Calder\'on-Vaillancourt Theorems) then extends to the manifold case.

An important example of a pseudodifferential operator is the
semiclassical Laplace--Beltrami operator $P(\h)=-\frac{\h^2}{2} \Delta_g$.
In local coordinates $(x; \xi)$ on $T^*M$, the operator can be expressed as
$\OpWh \big( |\xi|^2_g + \h (\sum_j b_j(x) \xi_j + c(x)) + \h^2 d(x) \big)$
for some functions $b_j,c,d$ on $M$.  In particular, its semiclassical
principal symbol is the function $|\xi |^2_g \in S^{2}(T^*M)$.
Similarly, the principal symbol of the Schr\"odinger operator 
$-\frac{\h^2}{2} \Delta_g + V(x)$ (with $V\in \Ci(M)$) is
$|\xi |^2_g + V(x) \in S^{2}(T^*M)$.

We will need the slightly more general class of symbols used in
\cite{DimassiSjostrand:SpecAsymp} (in fact, its adapation to $T^*M$ via
the partition of unity argument above).
In Euclidean space this is the class given for any $0 \leq \delta < 1/2$ by
\begin{eqnarray} \label{singsymbolclass}
S^{m}_{\delta}(\mathbb{R}^{2d}) &\defeq& \{a \in \Ci(\mathbb{R}^{2d} \times (0, 1] ): |\partial^{\alpha}_x\partial^{\beta}_{\xi} a| \leq C_{\alpha,\beta} \h^{-\delta|\alpha + \beta|}\langle\xi\rangle^{m-|\beta|} \}.
\end{eqnarray}
These symbols are allowed to oscillate more strongly when $\h\to 0$.
All the previous remarks regarding the case of $\delta=0$ transfer over in a straightforward manner.

\section{Ansatz and Quantum Birkhoff Normal Forms}

Let us remind ourselves that the variable $x \in \R^r$ is transverse to our submanifold $M$ and that our construction will be a quasimode which localizes in this variable.

\subsection{Separation of variables and the derivation of a hyperbolic operator} \label{s:sepvar}
At this point, we must fix our Fermi collar neighborhood (\ref{c:collar})
and therefore fix its corresponding width $\epsilon_1 > 0$
(the subscript is intentional, for reasons to be explained later).
In spite of our construction being entirely local, we will later show that
it extends globally therefore making our norm estimates and spectral-width
estimates independent of the chosen neighborhood $N_{\epsilon_1}(M)$,
after choosing our spectral parameters $\h \leq \h_0(M, \epsilon_1)$

For the problem of finding sufficiently thin spectral-width quasimodes
whose semiclassical measures place mass on the energy surface $S_{E_0}M$,
we consider the following natural ansatz in our coordinate system of
Lemma \ref{l:simp-op-vol}:
\bequ \label{e:preansatz}
\bPsih(\eta, w', x) = \Psih(x) \phih(\eta, w')\,.
\eequ

For the moment we choose $\phih$ to be an exact eigenfunction of $\DM$,
in that $-\h^2 \DM\phih = E_0 \phih$ for some $E_0 > 0$; hence, the problem
is to then choose $\psih(x)$ appropriately.  Our energies will have the form $\Eh = E_0 + f(\h)$, where $f(\h) = \cO(\h)$, so we need to
consider the quantity
\bequ
\left(-\h^2\DN- \Eh\right)\bPsih(\eta, w', x) = 0
\eequ
whilst keeping in mind the operator from Lemma \ref{l:simp-op-vol}.
The above expression then takes the form
\bequ
\left[ -\h^2 \left(\Delta_x + \left(\sum_{i=1}^r x_i\frac{\partial}{\partial x_i}\right)^2 + (n-1)\left(\sum_{i=1}^r x_i \frac{\partial}{\partial x_i}\right)\right) - \frac{E_0|x|^2}{1+|x|^2} - f(h) \right] \Psih(x) \phih(\eta, w') = 0.
\eequ
We therefore isolate the $x$-variable operator
\bequ \label{e:hyperbolic-op}
K_x(\h) \defeq -\h^2\left( \Delta_x + \left(\sum_{i=1}^r x_i\frac{\partial}{\partial x_i}\right)^2 + (n-1)\left(\sum_{i=1}^r x_i \frac{\partial}{\partial x_i}\right)\right) - \frac{E_0|x|^2}{1+|x|^2} - f(\h).
\eequ

The factorization of the volume element in Lemma~\ref{l:simp-op-vol}
shows we should consider $K_x(\h)$ acting on
$L^2\left(\R^r, (1+|x|^2)^{\frac{m-1}{2}} dx\right)$, where it will be
a formally symmetric operator (on smooth compactly supported functions)
by virtue of $\DN$ being symmetric.

\begin{lem}\label{l:ansatz-separates}
Suppose that $\{\phih\}_\h$ are $L^2(M)$-normalized quasimodes for $-\h^2\DM$,
with central energy $E_0$
and of width $\frac{c_1 \h}{\abs{\log\h}}$,
converging to the semiclassical measure $\musc$ on $S_{E_0}M$.
Suppose that $\{\Psih\}_\h$ are $L^2 \left(\R^r, (1+|x|^2)^{(m-1)/2} \right)$-normalized quasimodes for $K_x(\h)$ supported
in the collar $B(0,\epsilon_1)$
with central energy $0$ and of width $\frac{c_2 \h}{\abs{\log\h}}$, 
converging to a semiclassical measure $\sigma_{sc}$ on $T^*(\R^r)$ such that
$\sigma_{sc} \geq \nu \delta_0$; in other words, $\sigma$ gives mass at
least $\nu \in(0,1]$ to the point $(0;0)$.

Then $\{\bPsih\}_\h$ are normalized quasimodes for $-\h^2\DN$
with central energies
$\Eh = E_0 + f(\h)$ of width $\frac{(c_1+c_2) \h}{\abs{\log\h}}$, converging to a
semiclassical measure $\tilde{\mu}_{sc}$ giving mass at least $\nu$ to
$\iota_{*}\musc$ where $\iota\colon M\to N$ is the inclusion map.  In particular when $\nu = 1$ the quantum limit is exactly
$\iota_{*}\musc$.
\end{lem}
\begin{proof}
Since our collar $N_{\epsilon_1}(M)$ factors as a product
$B(0,\epsilon_1)\times M$ in the coordinate system, the identification of
the limit is clear.  It remains to verify that $\{\bPsih\}_\h$ are indeed
quasimodes of the stated width. Following the calculation above and Lemma \ref{l:simp-op-vol}, we have:
\begin{eqnarray}
&&\left(-\h^2\DN - \Eh\right)\bPsih = \left( \frac{-1}{1+\abs{x}^2} \h^2\DM + K_x(\h) + \frac{E_0 |x|^2}{1 + |x|^2} \right) \bPsih  \\
&&  =  \left(K_x(\hbar)\Psih\right)(x) \phih(\eta,w') +
  \Psih(x) \left( \frac{1}{1+\abs{x}^2} \right)\left(-\h^2\DM-\Eh\right)\phih(\eta, w').
\end{eqnarray}

Now
\begin{equation}
\norm{K_x(\hbar)\Psih }_{L^2\left(\R^r, (1+|x|^2)^{\frac{m-1}{2}} dx\right)} \leq \frac{c_2 \h}{\abs{\log\h}},
\end{equation}
\begin{equation}
\norm{\frac{1}{1+\abs{x}^2}\Psih(u)}_{L^2\left(\R^r, (1+|x|^2)^{\frac{m-1}{2}} dx\right)} \leq \norm{\Psih(x)}_{L^2\left(\R^r, (1+|x|^2)^{\frac{m-1}{2}} dx\right)} = 1\,,
\end{equation}
and
\begin{eqnarray}
&& \norm{\frac{1}{1+\abs{x}^2}\left(-\h^2\DM-\Eh\right)\phih(\eta, w') \Psih}_{L^2(N)} = \\
&& \norm{\frac{1}{1+\abs{x}^2} \Psih }_{L^2\left(\R^r, (1+|x|^2)^{(m-1)/2}\right)} \norm{\left(-\h^2\DM-\Eh\right)\phih(\eta, w') }_{L^2(M)} \leq \frac{c_1 \h}{\abs{\log\h}},
\end{eqnarray}
together giving the claim after applying Fubini's theorem thanks to the product form of the metric in Lemma \ref{l:simp-op-vol}.
\end{proof}

For the rest of the paper we will then construct transverse-to-$M$
quasimodes $\Psih$ of central energy $0$ for the operator
$K_x(\h)$, such that the quasimodes microlocalize to the delta measure
$\delta_0(x)$ or at least contains this measure.

Writing $\xi$ for the variable dual to $x$ on $\R^r$ and using the standard quantization, the total symbol
of $K_x(\h)$ as an $\h$-pseudodifferential operator
(see Section~\ref{s:hprelim}) is
\bequ
|\xi|^2 + \left(x\cdot\xi\right)^2 - n\h\left(x\cdot i\, \xi\right) - \frac{E_0|x|^2}{1+|x|^2} - f(\h)
\eequ
hence giving a semiclassical principal symbol of 
\bequ \label{prenormalsymbol}
\sigma(x;\xi) = |\xi|^2 + \left(x\cdot \xi\right)^2 - \frac{E_0|x|^2}{1+|x|^2};
\eequ
the reader is referred to \cite{Zworski:SemiclassicalAnalysis}
for further details on the symbol map.  Observe that $\sigma$ differs
from $\tilde{\sigma} := |\xi|^2 - E_0|x|^2$, the principal symbol for the
inverted harmonic oscillator, by terms of order $\cO(|x| + |\xi|)^4$
when $(x;\xi)$ is sufficiently close to $(0;0)$.
It follows that that $\sigma$ retains the non-degenerate critical
point of $\tilde{\sigma}$ at $(0;0)$.  The expression $\tilde{\sigma}$ is
a split quadratic form with eigenvalues $\{ \pm \mu_i \}_{i=1}^r$ where
$\mu_i > 0$.  A linear change of variables along with
$\xi \rightarrow \sqrt{E_0} \xi$ transforms $\tilde{\sigma}$ to
$\sum_{i=1}^r 2 \sqrt{E_0} \, x_i \xi_i$ whose Hamiltonian flow is the
model for a hyperbolic fixed point.

We will obtain our quasimodes $\Psih$ from the following ``quantized" version
of a classical Birkhoff normal form (see the beautiful text
\cite{AbrahamMarsden:FoundMechanics} by Abraham-Marsden for more on this topic) specifically
due to Iantchenko \cite{Iantchenko:BirhoffNormalForm98} but inspired by
the work of Sj\"ostrand \cite{Sjostrand:Resonances2D} in the context of
Hamiltonians with non-degenerate minima.  We state it in a slightly
different fashion which is more adapted to our upcoming calculations:

\begin{prop} \label{p:qbnf}
Let $P(\h)$ be a formally self-adjoint pseudodifferential operator, or rather a symmetric operator, acting on $\Cic(\R^r)$ with a $L^2(\R^r, \mu)$ structure for some positive measure $\mu$.  Consider its real Weyl symbol $\sum_{j=0}^{\infty} \h^j p_j$.  Suppose further that its semiclassical
principal symbol $p_0$ has a non-degenerate critical point and defines a
split quadratic form.  That is $p(0;0)=0, dp(0;0)=0,$ and $p''(0;0)$ is
non-degenerate but whose eigenvalues are
$\{ \pm \frac{\lambda_i}{2} \}_{i=1}^r$ where $\lambda_i > 0$.
Set $\overrightarrow{\lambda} \defeq (\lambda_1,\dots, \lambda_r)$.

For $\tN>0$ given, we can find neighborhoods $U, V$ of $(0;0)$ in
$T^*\R^r$ and a canonical transformation $\kappa:U \rightarrow V$, with
$\kappa(0;0)=(0;0)$, such that 
\bequ
\left( \sum_{j=0}^{\tN} \h^j p_j \right) \circ \kappa = q^{(\tN)} + r_{N+1} = \sum_{j=0}^{\tN} \h^j q_j(x;\xi)+ r_{\tN+1}(x;\xi;\h),
\eequ
where 
\bequ
q_0 = \sum_{i=1}^r \lambda_i x_i \xi_i + \sum_{l=2}^{\tN}   \left( \sum_{\alpha - \beta \in \cM, |\alpha| + |\beta| = l} c_{\alpha, \beta} x^{\alpha} \xi^{\beta} \right).
\eequ  
Here $\cM = \{\gamma \in \N^n :  \overrightarrow{\lambda} \cdot \gamma = 0 \}$ is the so-called resonance module for the Hamiltonian $q_{0,1}=\sum_i \lambda_i x_i \xi_i$ and $r_{\tN+1}(x;\xi;\h) = \cO\big((\h+|x| + |\xi|)^{\tN+1}\big)$.  The other terms $q_j$ have a similar form and for each $j$, $H_{q_{0,1}}q_j = 0$.  Here, $H_{q_{0,1}}$ is the Hamiltonian vector field associated to the function $q_{0,1}$ and takes the form $x \cdot \frac{\partial}{\partial x} - \xi \cdot \frac{\partial}{\partial \xi}$.

Moreover, we have a corresponding quantum Birkhoff normal form.  That is, for the given $\tN$ and $U,V$ as above there exists a microlocally unitary semiclassical Fourier integral operator $U(h): \Cic(\R^r) \rightarrow \mathcal{S}(\R^r)$ near $(0;0)$.  In other words, with respect to the given $L^2$ structure,
$\norm{U(\h)}_{L^2} = \norm{\chi_2(x; \h D_x) U(\h) \chi_1(x; \h D_x)}_{L^2}
  + \cO_{L^2}(\h^{\infty})$
for all microlocal cutoffs $\chi_1, \chi_2$ supported within $U$ and $V$
respectively and
\bequ
U_{\tN}^*(\h) P(\h) U_{\tN}(\h) = Q^{(\tN)}(\h) + R^{(\tN+1)}(\h).
\eequ
Here, $Q^{(\tN)}(\h)$ and $R^{(\tN+1)}(\h)$ are semiclassical pseudodifferential
operators with microlocal support in the indicated neighborhoods above with
respective Weyl symbols $q^{(\tN)}$ and $r_{\tN+1}$ where
$[ \OpWh \left(\sum_i \lambda_i x_i \xi_i\right), Q^{(\tN)}(\h) ] = 0.$

\end{prop}

\begin{rem}
In our setting where the energy $E = E_0$, the corresponding fixed point $(0;0)$ for our pre-normal form symbol in (\ref{prenormalsymbol}) yields expansion rates $\lambda_i = 2 \sqrt{E_0}$ after the discussed scaling.  For the sake of simplicity within our upcoming calculations and as the maximal Lyapunov exponent $\tilde{\lambda}_{E_0}$ on the energy shell $S^*_{E_0}N$ scales as $\sqrt{E_0} \tilde{\lambda}_1$, we consider the case of $\sqrt{E_0}=1$ for the remainder of our article.  We make note that the expansion rates transverse to $M$ are equal across the entire ambient manifold $N$, which is a feature of the constant curvature setting.
\end{rem}

\section{Propagation of a Gaussian wavepacket at the Hyperbolic Fixed Point}
\label{sec:propagate-wavepacket}

We begin this section with the following important remark:

\begin{rem}
This Section gives the higher-dimensional analogues of the results in
Section 5 of \cite{EswarathasanNonnenmacher:QuasimodeScarringSurfaces_preprint}.
The arguments are very similar, but we give considerable detail for
two reasons.  First, while the 
$r$-dimensional inverted harmonic oscillator obviously factors into a
product of $r$ one-dimensional oscillators (for initial conditions which
factor), the quantum Birkhoff normal form includes additional error terms,
and we need to track their behavior after conjugating by the semiclassical FIO in Proposition \ref{p:qbnf}.
Second, the argument for Corollary~\ref{c:sigmund} differs from previous
quasimode constructions in that we vary the submanifold $M$. 
Accordingly we must ensure that the spectral width, particularly the parameter $C>0$, of our quasimodes
does not depend on the submanifold $M$.
\end{rem}

\subsection{Ground states, squeezed states, and evolution}

\subsubsection{Preparing the Hamiltonian}
Consider the Gaussian
\begin{equation}\label{e:initial}
\Phi_0(x)\defeq \frac{1}{(\pi \h)^{r/4}}\, \prod_{i=1}^r \exp\big(-\frac{x_i^2}{2\h}\big)
\end{equation}
which is the $L^2$-normalized ground state of the harmonic oscillator on
$\R^r$ and has width $\h^{1/2}$ in each direction $x_i$.
We will evolve $\Phi_0$ through the reduced time-dependent
Schr\"odinger equation generated by the QNF operator $Q^{(\tN)}$ described in
Proposition~\ref{p:qbnf}, i.e. given the problem
\begin{equation}
 \begin{cases}
 i\h \partial_t \Phi^{(\tN)}_t &= Q^{(\tN)}(\h) \,\Phi^{(\tN)}_t, \, \text {where }  Q^{(\tN)} =\OpWh(q^{(\tN)}) \\
 \Phi^{(\tN)}_{t | t=0} &= \Phi_0,
\end{cases}
\end{equation}
we must analyze the evolution of $\Phi_0(x)$.  Therefore, our goal in this section is to describe the states
$$
\Phi^{(\tN)}_t = e^{-iQ^{(\tN)}t/\h}\,\Phi_0 \quad \text{for times }|t|\leq C(\tilde{\lambda}_1, \epsilon_2) \,|\log\h|
$$
where $C>0$ will be an explicit constant depending only on the maximal Lyapunov exponent $\tilde{\lambda}_1$ of $M$, in this case being $2$, and a small parameter $\epsilon_2$.  The index $\tN$ will be chosen later.

The Weyl symbol of $Q^{(\tN)}$ arising from Proposition \ref{p:qbnf} can be rewritten as
\begin{equation}
q^{(\tN)}(x, \xi; h) = \left(\sum_{i=1}^r q_{1,i}(h) x_i\xi_i \right) + \sum_{j=2}^{\tN} \left( \sum_{|\alpha|=|\beta|=j} q_{j}^{\alpha, \beta}(\h) x^\alpha \xi^\beta \right) \,,
\end{equation}
where the coefficients $q_{j}^{\alpha, \beta}(\h)$ expand like
$$
q_{j}^{\alpha, \beta}(\h)=\sum_{i=0}^{\tN} \h^i\,q_{j,i}^{\alpha,\beta}, \,\, \text{ for } q_{j,i}^{\alpha,\beta} \text{ a constant}
$$
for $j \geq 2$.  Moreover, for $j=1$, $q_{1,i}(h) = \lambda_i + \cO(h)$.  Here, we have used that $\lambda_i = 2$ for all $i$ and thus that the resonance
module is $\cM = \{\gamma \in \N^r :  \sum_i \gamma_i = 0 \} = \{0\}$.
The top order terms of the coefficients $q_{1,i}$ will play an important role in the
Schr\"odinger evolution around the unstable fixed point $(0;0)$.
In order to make the following analysis more transparent, we will
truncate our Weyl symbol $q^{(\tN)}$ to its top order quadratic terms and set
\begin{equation}
q_\qq = \sum_{i=1}^r \lambda_i x_i\xi_i \,.
\end{equation}
We label the operator whose Weyl symbol is $q_\qq$ as $Q_\qq(h)$.  Notice that this operator and its full symbol are \emph{independent of $\tN$}.

Recall from Proposition \ref{p:qbnf} that $\overrightarrow{\lambda}=(\lambda_1,\dots,\lambda_r)$.  The Schr\"odinger propagator for this special quadratic operator is easily expressed as
\begin{equation}\label{e:quad2}
U_\qq(t) \Phi_0 := \exp\big(-\frac{it}{\h} Q_\qq \big) \Phi_0 = \cD_{t \overrightarrow{\lambda}}\Phi_0\,, 
\end{equation}
where the unitary dilation operator $\cD_{\overrightarrow{\lambda}}: L^2(\R^r)\to L^2(\R^r)$ is given by
\begin{equation}
\cD_{\overrightarrow{\lambda}} u (x)\defeq \exp \big(-i \big[\sum_i \lambda_i \OpWh(x_i\xi_i) \big] /\h \big) u (x) = e^{-\sum_i\lambda_i/2} u (e^{-\lambda_1} x_1,\dots, e^{-\lambda_r} x_r)\,.
\end{equation}
The states $\cD_{\overrightarrow{\lambda}} \Phi_0(x)$
and their generalizations are known as \emph{squeezed states}.
They naturally appear in related problems.  Two applications include
the Gutzwiller trace formula \cite{CombescureRobert:SemiclassicalSpreading}
and the pioneering work of Babich--Lazutkin \cite{Lazutkin:KAM} on the
construction of quasimodes concentrating on closed elliptic geodesics.
For more applications of these special states see
\cite{CombescureRobert:SemiclassicalSpreading} and the references therein.

For $\tN\geq 2$, the operator $Q^{(\tN)}$ includes a nonquadratic Hamiltonian $Q^{(\tN)}_\nq$ with its Weyl symbol taking the form
\bequ\label{e:q_nq}
q^{(\tN)}_\nq\defeq q^{(\tN)}-\left(\sum_{i=1}^r q_{1,i}(h) x_i\xi_i \right) = \sum_{j=2}^{\tN} \big( \sum_{|\alpha|=|\beta|=j} q_{j}^{\alpha, \beta}(\h) x^\alpha \xi^\beta \big)\,
\eequ
and 
\begin{equation}
\tilde{q}_\qq = \left(\sum_{i=1}^r q_{1,i}(h) x_i\xi_i \right) - q_\qq,
\end{equation}
whose Weyl quantization we label as $\tilde{Q}_\qq$.  Note that as a symbol, $\tilde{q}_\qq = \mathcal{O}(\h)$ and that 
\begin{equation}
q^{(\tN)} = q_\qq + \tilde{q}_\qq + q^{(\tN)}_\nq.
\end{equation}  
Unfortunately, the second and third terms above cannot be ignored and a large part of our
construction will be in analyzing their contributions to the evolution of
$\Phi_0$.  However, the resonance condition given in
Proposition \ref{p:qbnf} greatly simplifies the actions of $\tilde{Q}_\qq$ and $Q^{(\tN)}_\nq$ as it ensures that this operator commutes with $Q_\qq$.


\medskip

\subsubsection{Squeezed excited states} We now recall some basic facts concerning the standard $r$-dimensional quantum harmonic oscillator $\sum_i (\h D_{x_i})^2+X_i^2$ where $X_i$ is the operator which multiplies by $x_i$.  Our initial state $\Phi_0 = \prod_{i=1}^r\varphi_0(x_i)$ is also the ground state of this operator, where the individual functions $\varphi_0(x_i) = \frac{1}{(\pi\h)^{1/4}} e^{-\frac{x_i^2}{2\h}}$ are themselves ground states of the 1-dimensional quantum harmonic oscillators $(\h D_{x_i})^2 + X_i^2$.

Let us call $(\varphi_m)_{m\geq 1}$ the 1-dimensional $m$-th excited states in the variable $x_i$, which are obtained by iteratively applying to $\phi_0$ the ``raising operator'' $a_i^*\defeq \OpWh(\frac{x_i-i\xi_i}{\sqrt{2\h}})$ and $L^2$-normalizing: 
\bequ\label{e:excited}
\varphi_m = \frac{(a^*)^m}{\sqrt{m!}} \varphi_0 \Longrightarrow \varphi_m(x_i) = \frac{1}{(\pi \h)^{1/4}2^{m/2}\sqrt{m!}}\,H_m(x_i/\h^{1/2})\, e^{-\frac{x_i^2}{2\h}}\,,
\eequ
where $H_m(\cdot)$ is the $m$-th Hermite polynomial.  We also have the dual
lowering operators $a_i \defeq \OpWh(\frac{x_i+i\xi_i}{\sqrt{2\h}})$
which satisfy the similar relation $a_i \phi_m(x_i) = \sqrt{m} \phi_{m-1}(x_i)$.

Now, given a $r$-vector of $1$-dimensional excited states
$(\varphi_{m_1},\dots, \varphi_{m_r})$, we can form the analogous $r$-\emph{dimensional excited state}
\begin{equation}
\Phi_{m_1,\dots, m_r}(x) \defeq \prod_{i=1}^r \varphi_{m_i}(x_i).
\end{equation}
This function continues to have $L^2$-norm 1, and by applying the unitary
dilation operator $\cD_{\overrightarrow{\lambda}}$ we obtain an
$L^2$-normalized \emph{squeezed excited state}
$\cD_{\overrightarrow{\lambda}} \Phi_{m_1,\dots, m_r}(x)$.
These are essentially products of unitarily scaled Gaussians in one variable
decorated by products of scaled polynomials.  The main property exhibited by
$\cD_{\overrightarrow{\lambda}} \Phi_{m_1,\dots, m_r}$ that we will use are its concentration properties which are similar to those of $\cD_{\overrightarrow{\lambda}} \Phi_{0}$.

\subsubsection{Expansion around the fixed point}

The following result is inspired by the work of Combescure--Robert
\cite{CombescureRobert:SemiclassicalSpreading} and its proof is an
$r$-dimensional version of Proposition 4.9 in
\cite{EswarathasanNonnenmacher:QuasimodeScarringSurfaces_preprint}.
Hence, we only provide the necessary details and leave the remaining
elements to the reader.

\begin{prop} \label{prop:propag}
For every $l, \tN \in \N$, there exists a constant $C_{l,\tN} > 0$ and coefficients $c_{p}(t, \h)\in\C$, which are polynomials in $t$ and $\h$, such that the following estimate holds for any $\h\in (0,1]$:
\begin{equation} \label{ineq:remest}
\forall t\in\R,\qquad \norm{e^{itQ^{(N)}/\h} \Phi_0 - \cD_{t\overrightarrow{\lambda}}\Phi_{0} - \sum_{p=1}^{l} c_{p}(t,\h) \cD_{t\overrightarrow{\lambda}} \big( \sum_{k_p}  d_{k_p} \Phi_{G(k_p)} \big)} \leq C_{l,N} \,(|t|\h)^{l+1}\,.
\end{equation}
The coefficients $c_{p}(t,\h)$ are $\mathcal{O}(|t|^l h^p)$.  The integer indices $k_p \leq C_{p, \tN}$, $d_{k_p} > 0$ are constants,
and $G(p) \in \N^n$ indexes the excited states.
Furthermore, the number of terms in the sums indexed by $k_p$ is less than some function depending only on $\tN$.
\end{prop}
\begin{rem}
This proposition will play a crucial role in determining the microlocal
concentration of our future constructed quasimode.  In order to show that
the spectral width of our quasimodes is uniform in how we vary our submanifold $M$,
we will need to keep track of the dependence of certain constants on the
parameters $l$ and $\tN$.  We will ultimately verify that any $l\geq 2$ is
sufficient for our purposes but that $\tN$ may have to be large,
specifically we shall require $(\tN+1)\epsilon_2 /3 > 1$ where
$\epsilon_2$ will appear in our Ehrenfest time.
\end{rem}

\begin{proof}
As in \cite{CombescureRobert:SemiclassicalSpreading}, we would like to
show that the full evolved state
$\Phi^{(\tN)}_t = U(t)\Phi_0 = e^{-iQ^{(\tN)}t/\h}\,\Phi_0$ is sufficiently
approximated by
$U_\qq(t) \Phi_0 = \cD_{\overrightarrow{\lambda}t} \Phi_{0}(x)$,
modulo a large sum of $r$-dimensional excited states whose $L^2$-norm
is sufficiently small (bounded by an appropriate power of $\h$).

The method of approximation is via the so-called Dyson expansion of
Duhamel's formula
\bequ
U(t) - U_\qq(t) = \frac{1}{ih}\int_0^t U(t-t_1) \left( \tilde{Q}_\qq + Q^{(\tN)}_\nq \right) U_\qq(t_1) \, dt_1,
\eequ
which corresponds to the choice $l=1$.  Using this formula directly leads
to the bound $\norm{U(t) \Phi_0 - U_\qq(t) \Phi_0} \leq C_{1, \tN} t\h$
for all $t\in\R$, $\h\in[0,1)$ (we will later see that
$t \leq C(\tilde{\lambda}_1, \epsilon_2) |\log \h|$).  This is not sufficient
since we want future error estimates to be $\cO(h^{1+\delta})$ for some $\delta>0$.

Accordingly we iterate the formula $l>1$ times to obtain
\begin{multline*}
U(t) - U_\qq(t) = \\ 
\sum_{j=1}^{l} \frac{1}{(i\h)^j} \int_{0}^t \int_{t_1}^t \dots \int_{t_{j-1}}^t U_\qq(t-t_j) \left( \tilde{Q}_\qq + Q^{(\tN)}_\nq \right)
U_\qq(t_j - t_{j-1})\left( \tilde{Q}_\qq + Q^{(\tN)}_\nq \right)\\
\cdots \left( \tilde{Q}_\qq + Q^{(\tN)}_\nq \right) U_\qq(t_{1})  \, dt_1 \dots dt_j\\
+ \frac{1}{(i\h)^{l+1}} \int_{0}^t \int_{t_1}^t \dots \int_{t_{l}}^t
U(t-t_{l+1}) \left( \tilde{Q}_\qq + Q^{(\tN)}_\nq \right) 
U_\qq(t_{l+1} - t_{l}) \\
\left( \tilde{Q}_\qq + Q^{(\tN)}_\nq \right)\cdots \left( \tilde{Q}_\qq + Q^{(\tN)}_\nq \right) U_\qq(t_{1})  \, dt_1 \dots dt_{l+1}.
\end{multline*}
To simplify the notation, we shorten the last term to $R_l^{(N)}(t,\h)$.

A crucial fact involving the quantum Birkhoff normal form in
Proposition~\ref{p:qbnf} is that $[Q^{(\tN)}_\nq, Q^{(\tN)}_\qq]=0$ and that $[\tilde{Q}_\qq, Q_\qq]=0$.
This implies $Q^{(\tN)}_\nq$ and $\tilde{Q}_\qq$ also commute with $U_\qq$ by functional calculus, effectively giving
us an exact form of Egorov's Theorem when we apply the quadratic evolution to our resonant
Hamiltonians: $(U_\qq)^*Q^{(\tN)}_\nq U_\qq = Q^{(\tN)}_\nq$ since the
Weyl symbols of $Q^{(\tN)}_\nq$ and $\tilde{Q}_\qq$ are resonant functions under the action
of the Hamiltonian vector field of $Q_\qq^{(\tN)}$.  For proofs of Egorov's Theorem in this case, see
\cite{CombescureRobert:SemiclassicalSpreading, Zworski:SemiclassicalAnalysis}.
This commutativity leads us to the simpler expression
\begin{multline}\label{e:dyson-l}
U(t) - U_\qq(t)=\sum_{p=1}^{l} \frac{t^p}{p!(i\h)^p} U_\qq(t) \left( \tilde{Q}_\qq + Q^{(\tN)}_\nq \right)^p \\
+ \frac{1}{(i\h)^{l+1}} \int_{0}^t \frac{t_l^{l}}{l!} U(t-t_{l+1})\,
U_\qq(t_{l+1})\,\left( \tilde{Q}_\qq + Q^{(\tN)}_\nq \right)^{l+1}\,dt_{l+1}\,.
\end{multline}

It is helpful to understand the explicit action of $Q^{(N)}_\nq$ on our
intial state $\Phi_0$.
As
\begin{equation}
((x+y)/2)^\alpha\xi^\beta =
   \prod_{i=1}^r\big( \sum_{\gamma_i}^{\alpha_i}
               \binom{\alpha_i}{\alpha_i - \gamma_i}
               \left(\frac{x_i}{2}\right)^{\alpha_i - \gamma_i}
               \xi_i^{\beta_i} \left(\frac{y_i}{2}\right)^{\gamma_i} \big)\,,
\end{equation}
we see that 
\bequ \label{monomialops}
\OpWh(x^\alpha \xi^\beta) = \prod_{i=1}^r\bigg( \sum_{\gamma_i=0}^{\alpha_i} \binom{\alpha_i}{\alpha_i - \gamma_i} \left(\frac{x_i}{2}\right)^{\alpha_i - \gamma_i} (\h D_{x_i})^{\beta_i} \left(\frac{x_i}{2}\right)^{\gamma_i} \bigg).
\eequ
Using that $X_i = \sqrt{\frac{\h}{2}}(a_i^* + a_i)$ and $\h D_{x_i} = i\sqrt{\frac{\h}{2}}(a_i^* - a_i)$, where $a_i^*$ and $a_i$ are the raising and lowering operators in the $x_i$ variable, we find that (\ref{monomialops}) reduces to
\begin{align}
\OpWh(x^\alpha \xi^\beta) &= \prod_{i=1}^r \OpWh(x_i^{\alpha_i} \xi_i^{\beta_i}) \\
&= \prod_{i=1}^r\bigg( \sum_{\gamma_i=0}^{\alpha_i} \binom{\alpha_i}{\alpha_i - \gamma_i}\left(\frac{1}{2}\right)^{\alpha_i} \left(\frac{\h}{2}\right)^{(\alpha_i + \beta_i)/2} i^{\beta_i} (a_i^* + a_i)^{\alpha_i - \gamma_i} (a_i^* - a_i)^{\beta_i} (a_i^* + a_i)^{\gamma_i} \bigg).
\end{align}
Here we have used composition formulae for the Weyl quantization and the Moyal product for the symbols in disjoint variables.  Keeping in mind the action of the raising and lowering operators on $\Phi_0$, it follows that 
\begin{eqnarray}
&& Q^{(\tN)}_\nq \Phi_0 = \sum_{j=2}^{\tN} \sum_{|\alpha|= |\beta| = j} q_j^{\alpha, \beta}(\h) \OpWh(x^\alpha \xi^\beta) \Phi_0 \\ 
&& = \sum_{j=2}^{\tN} \sum_{|\alpha| = |\beta| = j}^{(P^j_r)^2} q_j^{\alpha, \beta}(\h) \, \cO(\h^{(|\alpha| + |\beta|)/2}) \, \big( \prod_{i=1}^r \sum_{k=0}^{\lceil \frac{2^{\alpha_i + \beta_i}+1}{2} \rceil (\alpha_i + 1)} c_k^{\alpha_i, \beta_i} \phi_k(x_i)\big)  \\
&& = \sum_{j=2}^{\tN} \cO(\h^j) \bigg( q_j^{\alpha, \beta}(\h) \,  \prod_{i=1}^r \sum_{k=0}^{\lceil\frac{2^{\alpha_i + \beta_i}+1}{2}\rceil (\alpha_i + 1)} c_k^{\alpha_i, \beta_i} \phi_k(x_i)\bigg) \\
&& = \sum_{j=2}^{\tN} \cO(\h^j) \bigg( q_j^{\alpha, \beta}(\h) \,\sum_{k=0}^{\prod_{i=1}^r \lceil\frac{2^{\alpha_i + \beta_i}+1}{2}\rceil (\alpha_i + 1)} c(\alpha, \beta, k) \Phi_{F(\alpha, \beta, k)} \bigg) 
\end{eqnarray}
where $P^j_r$ is the partition function and $\Phi_F$ is an
$r$-dimensional excited state with $F \in \N^r$ being a function of
$\alpha, \beta,$ and $k$.  We have also used the fact that for odd
(respectively, even) $\alpha_i + \beta_i$, the operator
$(a_i^* + a_i)^{\alpha_i - \gamma_i} (a_i^* - a_i)^{\beta_i} (a_i^* + a_i)^{\gamma_i}$
yields a sum of odd (respectively, even) indexed excited states when
applied to $\phi_0(x_i)$.  This is due to the fact that any word of
length $W$ consisting of $a_i^*$ and $a_i$ is a sum of words in
``normal ordering" $(a_i)^m (a_i^*)^n$ 
(after using that $[a_i^*, a_i] =1$) each of whose length is equal in
parity to $W$.  Hence, after using that $\tilde{Q}_\qq \Phi_0 = \mathcal{O}(\h^2)$, 
\begin{eqnarray}
(\tilde{Q}_\qq + Q_\nq^{(\tN)})^l \Phi_0 = \sum_{j=2l}^{\tN l} \cO(\h^j) \big( \sum_k^{L(\tN,l)} d_k \, \Phi_{G(k)} \big) 
\end{eqnarray}
where $d_k > 0$ are constants, $L(\tN,l)>0$ is a large (but computable) constant, and $G(k) \in \N^n$.  Moreover the $L^2$ norm of each term of (\ref{e:dyson-l}) applied to our ground state, keeping in mind that $U_\qq(t)$ is unitary, is
\begin{equation}
\norm{\frac{t^p}{p!(i\h)^p} U_\qq(t)\big( \tilde{Q}_\qq + Q^{(\tN)}_\nq\big)^p \Phi_0} \leq C_{p,\tN, r} (|t|^p \h^p).
\end{equation}
The constant $C_{p,\tN}$ grows exponentially in $\tN$ but since $\tN$ is
independent of $\h$ this is not an issue.

We now return to the remainder term $R_l^{(N)}(t,\h)$ from \eqref{e:dyson-l}.
It yields 
\bequ\label{e:remain-l}
\norm{\frac{1}{(i\h)^{l+1}} \int_{0}^t \frac{t_l^{l}}{l!} U(t-t_{l+1})\,U_\qq(t_{l+1})\, (\tilde{Q}_\qq +  Q^{(\tN)}_\nq )^{l+1}\Phi_0 \,dt_{l+1}}\leq C_{l,\tN, r}\,(|t|\h)^{l+1}\,.
\eequ
Using these estimates, we group the terms \eqref{e:dyson-l} in
increasing powers of $\h$.  Although the notation is tedious, it useful
to write the general expression of the $p$-th term in the sum
\eqref{e:dyson-l} (removing the factor $\frac{t^p}{p!(i\h)^p}$):
\begin{equation}
\mathcal{O}(\h^{2p}) \sum_{j_1,\dots, j_p}^{N, \dots, N} \, \, \sum_{|\alpha^1|=|\beta^1|=j_1, \dots, |\alpha^p|=|\beta^p|=j_p} q_{i_1}^{\alpha^1, \beta^1}(\h) \dots q_{i_p}^{\alpha^p, \beta^p}(\h) \, \OpWh(x^{\alpha^1}\xi^{\beta^1}) \dots \OpWh(x^{\alpha^p}\xi^{\beta^p}) \Phi_0,
\end{equation}
where $\cO(\h^{2p})$ denotes a function which is bounded above by a positive constant times $\h^{2p}$.  Beyond the principal term $\cD_{t \overrightarrow{\lambda}} \Phi_0$,
the Dyson expansion then takes the form
\bequ
\cD_{t \overrightarrow{\lambda}} \left[c_1(t,h) \big( \sum_{k_1}^{L(1,N)} d_{k_1} \, \Phi_{G(k_1)} \big) + c_2(t,h) \big( \sum_{k_2}^{L(2,N)} d_{k_2} \, \Phi_{G(k_2)} \big) + \dots + c_l(t,h)\big( \sum_{k_l}^{L(l,N)} d_{k_l} \, \Phi_{G(k_l)}\big) \right]
\eequ
where $c_p(t,h) = \cO(|t|^l \h^p)$.  

\end{proof}

\subsection{Microlocal support of the evolved state}

Fix a small $\vareps_2\in(0,1)$, For $\h\in (0,1/2]$ we define the
\emph{local Ehrenfest time}
\begin{equation}\label{e:Ehrenfest}
T_{\vareps_2}\defeq \frac{(1-\vareps_2)|\log\h|}{2\tilde{\lambda_1}}\,,
\end{equation}
where $\tilde{\lambda_1}$ is the maximal expansion rate amongst the individual rates $\{\lambda_i\}$; in our specific case, $\tilde{\lambda_1}=2$.

\begin{prop}\label{p:localization1}
Let $\tM > 0$ be some given power of $\h$.  Let $\Phi_t^{(N)} = e^{-iQ^{(\tN)}t/\h}\,\Phi_0$, with its expansion given in Proposition \ref{prop:propag}, and set $l = \tM$. Take $\Theta\in \Ci(T^*(\R^r))$ with $\Theta\equiv 1$ in a neighbourhood of $(0;0)$, and denote its rescaling by $\Theta_{\alpha}(x,\xi)\defeq \Theta(x/\alpha,\xi/\alpha)$. 
Then, for a normal form degree $\tN \in \N$ given in Proposition \ref{p:qbnf}, and $\epsilon_2 > 0$, there exists $C_{\tM,\tN, l}>0$ such that
\bequ
\norm{[\Oph(\Theta_{\h^{\varepsilon_2/3}})-I] \Phi_t^{(N)}}_{L^2} \leq
C_{\tM,\tN, l, r}\,\h^{M}\,,\qquad \h\in (0,1/2]\,,
\eequ
uniformly for times $t\in [-T_{\vareps_2},T_{\vareps_2}]$.
\end{prop}

\begin{proof}
This is a variant of
\cite[Prop.\ 4.22]{EswarathasanNonnenmacher:QuasimodeScarringSurfaces_preprint}.  We begin with cutoffs given by a product, choosing
$\theta_{i} \in \Cic([-2,2]^2)$ for each
$i=1,\dots, r$ such that $\theta_i = 1$ in $[-1,1]^2$, and scaling them as
$\theta_{i, \alpha}(x,\xi)\defeq \theta_i(x/\alpha,\xi/\alpha)$.
It was shown in
\cite{EswarathasanNonnenmacher:QuasimodeScarringSurfaces_preprint} that
for any index $m$ there exists $C_m>0$ in a bounded range such that
$$
\norm{\left[\OpWh(\theta_{i, \alpha})-I \right] \cD_{t\lambda_i}\varphi_{m,i}}_{L^2} \leq C_m \h^{\tM} \,,
$$
uniformly for $|t|\leq T_{\vareps_2}$ and width $\alpha\geq \h^{\vareps_2/3}$. 

We propagate the product and use Proposition~\ref{prop:propag}
along with the fact that each expression $\sum_{k_p}^{L(p,\tN)} d_{k_p} \Phi_{G(k_p)}$
is a sum of terms each of which is a product of scaled Gaussians
$\phi_{0}$ in disjoint variables (each of uniform width
$e^{t \tilde{\lambda}_1} \h^{1/2} \leq \h^{\epsilon_2/2}$ by \eqref{e:Ehrenfest}))
multiplied by polynomial factors.
As we have a product structure in our evolved state given by
Proposition~\ref{prop:propag}, we have the same estimate for the
individual terms:
\bequ
\norm{[\Oph\left(\prod_{i=1}^r \theta_{i, \alpha}\right)-I] \cD_{t\overrightarrow{\lambda}}\Phi_{G(p)}}_{L^2} \leq C_{\tM, \tN,r} \h^M
\eequ
uniformly for $|t|\leq T_{\vareps_2}$, multi-indices $G(p)$ in a bounded range, and the same width $\alpha\geq \h^{\vareps_2/3}$ in all variables where $C_{N,p} >0$. As the number of terms in the sum (\ref{ineq:remest}) depends only $l$ and not on $\h$, this estimate holds for $\Phi^{(N)}_t$ as well but now with a constant $C'_{l,\tN,r} > 0$.  Since $R_l^{(N)}(t,\h) = \cO(\h^{l+1 - \delta})$ for some $\delta > 0$ arising from the fact that $t \in [-T_{\epsilon_2},T_{\epsilon_2}]$, we set $l = M$ as the number of interations in the Dyson expansion which in turns give us the following estimate for the full state $\Phi^{(N)}_t$:
\bequ \label{i:supportprodform}
\left\|[\Oph\left(\prod_{i=1}^r \theta_{i, \alpha}\right)-I] \Phi^{(N)}_t \right\|_{L^2} \leq C'_{\tM, \tN,r} \h^M.
\eequ

Now, consider a non-product form cutoff $\Theta_{\alpha} \in \Cic(\R^{2r})$ equal to 1 on the support of $\prod_{i=1}^r \theta_{i}$.  Then 
\bequ 
\Theta_{\alpha} - 1 = \left(1 - \Theta_{\alpha}\right)\left(\prod_{i=1}^r\theta_{i, \alpha} - 1 \right).
\eequ
Taking $\alpha \geq \h^{\epsilon_2/3}$, $\epsilon_2/3 >1/2$ so that
the semiclassical symbol class $S^{0}_{\epsilon_2/3}$
(see \eqref{singsymbolclass}) continues to have expansions in terms of
increasing powers of $\h$, and using the support properties of
$\Theta_{\alpha}$ yields pseudodifferential cutoffs of the form 
\bequ
\OpWh\left(\Theta_{\alpha} - 1 \right) = \OpWh\left(1 - \Theta_{\alpha}\right)
\OpWh\left(\prod_{i=1}^r\theta_{i, \alpha} - 1 \right) + \cO_{L^2}(\h^{\infty}).
\eequ
Our proof is complete after applying this operator to $\Phi^{(N)}_t$ and using (\ref{i:supportprodform}).
\end{proof}

\section{A log-scale quasimode for the Birkhoff normal form}

A straightforward calculuation gives
$\norm{Q^{(N)}(\h) \Phi_0} =
       \left( \prod_{i=1}^r \sqrt{2}\lambda_i \right) \h
       + \cO_{N}(\h^2)$
after recalling that $\Phi_0$ is the standard Gaussian of equation \eqref{e:initial}.
Therefore applying the unitary Fourier integral operators
of Proposition~\ref{p:qbnf} to the state constructed so far
produces a quasimode $\bPsih$ for equation \eqref{e:preansatz} with a spectral width which is too large.

To get a narrower quasimode we use the usual time-averaging proceudre
originally due to Vergini--Schneider \cite{VerginiSchneider:MatrixCoeffScars}
and also employed in the work \cite{EswarathasanNonnenmacher:QuasimodeScarringSurfaces_preprint}.

Let $T>0$ be an averaging time which will be chosen later, fix a weight function
$\chi\in \Cic((-1,1),[0,1])$ and its rescaled version
$\chi_T(t)\defeq \chi(t/T)$.  Our transverse quasimode will be
\bequ\label{eq:def-qmode}
\tilde{\Phi}^{(\tN)}_{\chi_T,\h} \defeq  \int_{\R}\chi_T(t)\, e^{it(E_{\h} - 1)/\h}\Phi^{(\tN)}_t\,dt
\eequ
where $E_{\h} = 1 + f(\h)$.  Note that this state is not yet normalized. In order to compute its spectral width, we will first need to compute its $L^2$ norm. 
\begin{lem}\label{l:norm-qmode}
Let $C>0$ be a constant which will be chosen later.  For the semiclassically large averaging time $1\leq T=T_\h\leq C|\log\h|$ the square norm of our state
$\tilde{\Phi}^{(\tN)}_{\chi_T,\h}$ satisfies
$$
\norm{\tilde{\Phi}^{(\tN)}_{\chi_T,\h}}^2 = T\,S(\lambda, f(\h)/\h, r)\,\norm{\chi}_{L^2}^2\, \big(1 + \cO_{\tN,r}(1/T)\big)\,,
$$ 
where $S(\bullet, \bullet)$ is a positive (and explicit) function $\h$
small enough, and $\lambda$ is the vector of expansion rates transverse
to $M$ as seen in Proposition \ref{p:qbnf}. 
\end{lem}

\begin{proof}
Although we obtained an explicit expression for the Dyson series
\eqref{ineq:remest}, we prefer here the slightly less explicit operator
equation \eqref{e:dyson-l} in order to reduce our calculations to those in
the proof of
\cite[Lem.\ 5.4]{EswarathasanNonnenmacher:QuasimodeScarringSurfaces_preprint}.

We begin with the representation of the evolved state as
\bequ
\Phi^{(\tN)}_t = \left( U_\qq(t) + \sum_{p=1}^{l} \frac{t^p}{p!(i\h)^p} U_\qq(t) (Q^{(\tN)}_\nq)^p
+ R_l \right) \Phi_0
\eequ
for some $l \in \N$ to be determined later.
The norm squared of the averaged quasimode is then
\bequ\label{e:norm}
\norm{\tilde{\Phi}^{(\tN)}_{\chi_T,E_\h}}^2 = \int \int e^{-i(t'-t)f(\h)/\h}\langle \Phi^{(\tN)}_{t'},\Phi^{(\tN)}_t\rangle\,\chi_T(t')\,\chi_T(t)\,dt\,dt'\,.
\eequ
The key is then approximating the overlaps
$\langle \cD_{t \overrightarrow{\lambda}}(Q^{(\tN)}_\nq)^p \Phi_0,  \cD_{t' \overrightarrow{\lambda}}(Q^{(\tN)}_\nq)^{p'} \Phi_0 \rangle$ \break $= \langle (Q^{(\tN)}_\nq)^p \Phi_0,  \cD_{(t'-t) \overrightarrow{\lambda}}(Q^{(\tN)}_\nq)^{p'} \Phi_0 \rangle$.

Using again the resonance condition on the Weyl symbols we have that
$(Q^{(\tN)}_\nq)^p$ is a power of a symmetric operator and that
$(Q^{(\tN)}_\nq)^p$ commutes with $\cD_{t \overrightarrow{\lambda}}$.
From these
\bequ
\langle (Q^{(\tN)}_\nq)^p \Phi_0,  \cD_{(t'-t) \overrightarrow{\lambda}}(Q^{(\tN)}_{nq})^{p'} \Phi_0 \rangle = \langle (Q^{(\tN)}_\nq)^{p+p'} \Phi_0,  \cD_{(t'-t) \overrightarrow{\lambda}} \Phi_0  \rangle\,.
\eequ
Since for our Euclidean quantization
$\OpWh(x^\alpha \xi^\beta) = \prod_{i=1}^r\OpWh(x_i^{\alpha_i}\xi_i^{\beta_i})$,
the operator $(Q^{(\tN)}_\nq)^p$ maintains a product-type form into sums of
differential operators in disjoint variables, it suffices to estimate
products of the form
\bequ
\prod_{i=1}^r \left\langle \sum_{m_i \geq 0}^{K(\tN,p,p')} c_{m_i} \phi_{m_i},  \cD_{(t'-t)\lambda_i} \phi_0 \right\rangle_{\R_{x_i}},
\eequ
where $K(\tN,p,p') > 0$. 

A straightforward calculation shows
$$ \la \phi_0, \cD_{(t'-t)\lambda_i} \phi_0 \ra_{\R_{x_i}} = \frac{1}{\sqrt{\cosh{\lambda_i(t'-t)}}}.$$
A similar identity for excited states is developed in
\cite[Sec.\ 5.1]{EswarathasanNonnenmacher:QuasimodeScarringSurfaces_preprint}
gives for each $m \in \N$ a constant $C_m > 0$ such that 
\bequ
\left| \frac{\la \phi_m, \cD_{(t'-t)\lambda_i} \phi_0 \ra_{\R_{x_i}}}{\la \phi_0, \cD_{(t'-t)\lambda_i} \phi_0 \ra_{\R_{x_i}}} \right| \leq C_m \, \text { uniformly in } t,t' \in \R.
\eequ
(in fact $C_m = 0$ for odd $m$ since in that case we are integrating an
odd function against an even function).

Returning to equation \eqref{e:norm}, the term arising from the
quadratic operator $U_\qq(t)$ which corresponds to the case $m=0$ takes
the form 
\bequ
I_{0}=\int e^{-i(t'-t)f(\h)/\h} \left( \prod_{i=1}^r \, \la\varphi_0,\cD_{(t-t')q^1}\varphi_0\ra_{\mathbb{R}_{x_i}} \right) \, \chi_T(t')\,\chi_T(t)\, dt\,dt'
\eequ
\bequ
= \int e^{-i(t'-t)f(\h)/\h} \left( \prod_{i=1}^r \, \frac{1}{\sqrt{\cosh \left(q_{1,i}(\h)(t'-t) \right)}} \right) \, \chi_T(t')\,\chi_T(t)\, dt\,dt'\,.
\eequ
A similar expression was evaluated in
\cite[Sec.\ 5.1]{EswarathasanNonnenmacher:QuasimodeScarringSurfaces_preprint},
giving
\bequ \label{e:I_0}
I_{0}(\h)=T\,S(\overrightarrow{\lambda},f(\h)/\h, r)\,\norm{\chi}_{L^2}^2\, (1 + \cO_r(1/T)).
\eequ
It is also shown there that the correction terms arising from $m>0$ 
(and the remainder $R_l$) are bounded above by $\tilde{C}_{m} \h T^l I_0(0)$
when $\h$ is small enough.  We have taken $T \sim |\log \h|$
so the correction terms are $\cO(\h^{\delta})$ for some $\delta>0$ and are
therefore lower order than the constant appearing in \eqref{e:I_0}.

For this we need to bound $I_0(\h)$, that is
$S(\overrightarrow{\lambda},f(\h)/\h, r)$, above and below.  We would like
to do this uniformly for $\h$ small enough, as this uniformity feeds into the
argument for Corollary~\ref{c:sigmund}.  As $f(\h)/\h = \cO(1)$ and
$S(\overrightarrow{\lambda},f(\h)/\h, r)$ is the 1-dimensional Fourier transform
of a non-zero Schwartz function (for all $\h$),
the positivity of $\norm{\tilde{\Phi}^{(\tN)}_{\chi_T,E_\h}}^2$ establishes that of $S$.
\end{proof}

We are can now define our penultimate normalized state 
\bequ
\Psi^{(N)}_{\chi_T,\h} \defeq
   \frac{\tilde{\Phi}^{(N)}_{\chi_T,\h}}{\norm{\tilde{\Phi}^{(N)}_{\chi_T,h}}}.
\eequ

\begin{cor}\label{c:localization2}
Given $\epsilon_2 > 0$ we may choose an averaging time $T = T_{\vareps_2} \leq C(\epsilon_2) |\log h|$ where $C(\epsilon_2)>0$ such that the normalized state
$\Psi^{(\tN)}_{\chi_T,h}$ is localized in the
$\h^{\vareps_2/3}$ neighbourood of $(0;0)$.  That is, for any
$\Theta\in \Cic(\R)$ with $\Theta\equiv 1$ in a fixed neighbourhood of
$(0;0)$, we have the estimate
\bequ
\norm{[\Oph(\Theta_{\h^{\vareps_2/3}})-I] \Psi^{(\tN)}_{\chi_T,\h}}_{L^2} = \cO_{\tN}(\h^\infty)\,.
\eequ
\end{cor}
\begin{proof}
This is immediate from Proposition~\ref{p:localization1} whilst keeping
in mind that $T_{\epsilon_2}$ goes to the Ehrenfest time \eqref{e:Ehrenfest}.
\end{proof}

\begin{prop} \label{p:logwidth}
For $T=T_{\epsilon_2}$, the spectral width of the state
$\Psi^{(\tN)}_{\chi_T,h}$ at energy $0$ is
\bequ
\norm{\left(Q^{(\tN)}(\h)-f(\h) \right)\Psi^{(\tN)}_{\chi_T, \h}}^{2} =
\frac{\h^2}{T^2_{\vareps_2}}\,\frac{\norm{\chi'}_{L^2}^2}{\norm{\chi}_{L^2}^2}\, (1 + \cO_{\tN,r}(1/|\log\h|)).
\eequ
Hence, $\{\Psi_{\chi_T,h}\}_{\h}$ yields a semiclassical measure which is invariant under the classical dynamics generated by $Q^{(\tN)}(h)$ and equal to $\delta_0(x) \in \mathcal{D}'(\R^r)$.
\end{prop}
\begin{proof}
Recalling that $\Psi_{\chi_T, \h}$ is defined by (\ref{eq:def-qmode}), we perform the following calculation:
\begin{align*}
\left(Q^{(\tN)}(h)-f(\h) \right)\Psi_{\chi_T, h}&= \int_{\R}\chi_T(t)\,\left(Q^{(\tN)}(h)-f(\h) \right)\,e^{-it(Q^{(\tN)}(h) - f(\h))/\h}\Phi_0\,dt\\
&= \int_{\R}\chi_T(t)\,i\h\partial_t \, (e^{-it(Q^{(\tN)}(h) - f(\h))/\h})\Phi_0\,dt\\
&= -i\h \int_{\R}(\partial_t\chi_T(t))\,e^{-it(Q^{(\tN)}(h) - f(\h))/\h}\Phi_0\,dt\,.
\end{align*}
The norm of the last integral was essentially computed in
Lemma~\ref{l:norm-qmode} except our cutoff in time is now
$-i\h \, \chi'(t/T)/T$.  A division by the asymptotic formula
in the same Lemma for $\norm{\Psi_{\chi_T, \h}}$ finishes the first
statement of our proposition.  

The second statement follows from our quasimode having spectral width
which is clearly $o(\h)$ and therefore yielding an invariant semiclassical
measure under $\exp(tH_{q^{(\tN)}})$ \cite{Zworski:SemiclassicalAnalysis}.
The microlocal support statement from Corollary~\ref{c:localization2}
tells us that this can only be $\delta_0$.
\end{proof}

The spectral width of the normalized mode $\tilde{\Psi}^{(\tN)}_{\chi_T, h}$
then takes the form
$$
F(\h)=\frac{\h}{T_{\vareps_2}}\frac{\norm{\chi'}_{L^2}}{\norm{\chi}_{L^2}}(1 + \cO_{\tN,r}(1/|\log\h|))\,,
$$
We note that the ratio $\frac{\norm{\chi'}_{L^2}}{\norm{\chi}_{L^2}}$
is essentially a universal constant.  Since
\cite[Lem.\ 5.16]{EswarathasanNonnenmacher:QuasimodeScarringSurfaces_preprint}
determined the infimum of these ratios to be $\pi/2$, we make the slight improvement of choosing $\chi$ for which the
ratio is $(1+\epsilon_2)\pi/2$.  The resulting spectral width is therefore
\begin{eqnarray*}
F(\h) = & \pi \tilde{\lambda}_1 \frac{1 + \epsilon_2}{1-\epsilon_2}\frac{\h}{|\log h|} + \cO_{\tN,r}\left(\frac{h}{|\log h|^2} \right) \\
\leq & \pi \tilde{\lambda}_1 (1 + 3\epsilon_2)\frac{\h}{|\log h|} + \cO_{\tN,r}\left(\frac{h}{|\log h|^2} \right)
\end{eqnarray*}
where $\tilde{\lambda}_1 = 2$.

\section{A log-scale quasimode on $N$ and the proof of Theorem \ref{t:main}}

We now set $C = \pi \tilde{\lambda}_1 (1 + 3\epsilon_2)$.  Having constructed quasimodes of spectral width $\frac{C\h}{\abs{\log\h}}$
for our transverse dynamics in the Birkhoff Normal Form coordinates,
we would like to transport them into our
collar neighborhood $N_{\epsilon_1}(M)$ and plug them into our ansatz.
The necessary ingredient is cutting off $\Psi_{\chi_T, \h}$ in space.

We revert to our local Fermi-normal coordinate system
$(\eta, w', x)$ in the collar neighbourhood $N_{\epsilon_1}(M)$ determined in
Section~\ref{s:coords} where $(\eta, w')$ are local coordinates on $M$ and $x$
is our $r$-dimensional transverse variable.  Consider the energy $E_{\h} = 1 + f(\h)$.

By Lemma~\ref{l:ansatz-separates}, it suffices to analyze the behavior
of the result of transporting the quasimode of the previous section to
a quasimode for $K_x(\h)$ in the $x$ variable, which we proceed to do.

Let $\Upsilon \in \Cic(\R^r)$ be supported on
$[-\epsilon_1/2,\epsilon_1/2]^r$ and equal to 1 on
$[-\epsilon_1/3, \epsilon_1/3]^r$. Choose $\epsilon_2>0$ which is related to our semiclassical averaging time.  We set
$\psih = \Upsilon(x) U_{\tN}\left(\Psi^{(\tN)}_{\chi_T, h}\right)$ and set
\bequ \label{e:qmodeN}
\bPsih = \phih(\eta, w') \psih(x)\,.
\eequ

In the definition of $\psih$, $U_{\tN}(\h)$ is the microlocally unitary
FIO discussed in Proposition~\ref{p:qbnf} and
$\Psi^{(\tN)}_{\chi_T, \h}$ is as in equation \eqref{eq:def-qmode}.
Then for $\h \leq \h_0(\epsilon_1)$, $\psih$ will be supported in the range of the
transverse variable $u$ permitted in the collar neighbourhood making our
function $\bPsih$ is well-defined.  We need to verify that the cutoff $\Upsilon$ does not affect the norm,
spectral behavior, and concentration properties.

For the norm, Corollary~\ref{c:localization2} and properties
of Fourier integral operators associated to canonical transformations
(particularly those appearing Proposition~\ref{p:qbnf}) tell us that
for $\h \leq \h_0(\epsilon_1)$ small enough,
\bequ
\norm{\Upsilon U_{\tN}\left(\Psi^{(\tN)}_{\chi_T, \h}\right)}
       _{L^2\left(\R^r, (1+|x|^2)^{\frac{m-1}{2}} dx\right)} =
\norm{\Psi^{(\tN)}_{\chi_T, \h}} + \cO_{\tN}(\h^{\infty}) = 1 + \cO_{\tN}(\h^{\infty}).
\eequ

Note that the statement of Corollary~\ref{c:localization2} is exactly that
$\{\psih\}_{\h}$ concentrate at $0$ in the Birhkoff normal form coordinates.
Proposition~\ref{p:qbnf} demonstrates this concentration carries to concentration near $x=0$ for
$\{U_{\tN}\left(\tilde{\Psi}^{(\tN)}_{\chi_T, \h}\right)\}_{\h}$ and multiplying by
the cutoff $\Upsilon$ does not affect that for $\h \leq \h_0(\epsilon_1)$.

For the spectral width, we need to verify that applying the FIO $U_N$ and
multiplication by $\Upsilon$ have negligible effect.  Concerning
$U_{\tN}$, we have
\begin{eqnarray}
 & \norm{K_x(\h)\left( \Upsilon U_{\tN}\left(\Psi^{(\tN)}_{\chi_T, \h}\right)\right)}_{L^2\left(\R^r, (1+|x|^2)^{\frac{m-1}{2}} dx\right)} = \\
& \norm{U_{\tN}\left( \left(Q^{(\tN)}(\h)  + R_{\tN+1}(\h) \right) \, \Psi^{(\tN)}_{\chi_T, \h}\right)} +  \cO_{\tN,r}(\h^{\infty})
\end{eqnarray}
for $\h$ small enough.  Proposition \ref{p:logwidth} along with our optimized cutoffs $\chi$ then give:
\bequ
\norm{U_{\tN}\left( Q^{(\tN)}(\h)  \, \Psi^{(\tN)}_{\chi_T, h}\right)}
  = \pi \tilde{\lambda}_1 (1 + 3\epsilon_2)\frac{\h}{|\log \h|} 
    + \cO_{\tN,r}\left(\frac{\h}{|\log\h|^2}\right).
\eequ

The contribution from the remainder $R^{(\tN+1)}(\h)$ is small: a
second application of Corollary \ref{c:localization2} tells us that
$R^{(\tN+1)}(\h) \Psi^{(\tN)}_{\chi_T, \h} = R^{(\tN+1)}(\h) \Theta_{\h^{\epsilon_2/3}}(\h)\tilde{\Psi}^{(\tN)}_{\chi_T, \h} + \cO_{\tN,r}(\h^{\infty})$.
The composition calculus for Weyl quantizations gives
$R^{(\tN+1)}(\h) \Theta_{\h^{\epsilon_2/3}}(\h) = \OpWh \left( r_{\tN+1} \sharp \Theta_{\h^{\epsilon_2/3}} \right)$.  The Weyl symbol of this Moyal product
satisfies the bound $$\cO_{r}(\h^{(\tN+1)\epsilon_2/3})$$ after using the symbol
estimates in Proposition \ref{p:qbnf}.  Hence, letting $(\tN+1)\epsilon_2/3 > 1$ and recalling that $l \geq 1$ from the Dyson expansion estimate (\ref{ineq:remest}) guarantees that our
various remainder estimates are $o(\h/|\log \h|^2)$.

Finally, multiplication by $\Upsilon$ has no effect since the symbol of
$\Upsilon$ is the constant $1$ near $0$ and we have chosen $\h \leq \h_0(\epsilon)$ so that the quasimode is already concentrated near $0$ before the multiplication. This completes the proof of statement (1) of Theorem \ref{t:main}.

\begin{rem} \label{r:paramorder}
We record here the order in which the various parameters in the construction
have been chosen:
\begin{itemize}
\item Fix a codimension $r < n = \dim N$.
\item Given $M \subset N$ of codimension $r$, choose $\epsilon_1 > 0$ such that
the collar neighborhood $N_{\epsilon_1}(M)$ behaves as expected, including
in supporting a convenient coordinate system.
\item Choose $\epsilon_2>0$, giving the averaging time $T_{\epsilon_2}$.
\item Choose $\tN>0$ such that $\tN+1 > 3/\epsilon_2$.
\item Choose any $l>1$.
\item Determine $\h_0$ depending on all the previous choices such that if $0<\h<\h_0$ then the final state $\bPsih$ is microlocalized within the collar neighbourhood $N_{\epsilon_1}(M)$.
\end{itemize}
\end{rem}

The following lemma is a slight restatement of that in Section 6 of \cite{EswarathasanNonnenmacher:QuasimodeScarringSurfaces_preprint} and proves statement (2) of our Theorem \ref{t:main}.  We note that its proof is independent of whether the support of the invariant semiclassical measure $\mu_{sc}$ associated to $\{\bPsih\}_{\h}$ is smooth or not.  Therefore, we leave it to the interested reader for details.

\begin{lem} [Partial localization on invariant subsets] \label{l:partloc}
Let $\epsilon_2, \, \epsilon_3>0$ be given.  Let $E_0>0$ be a regular energy level of $-h^2 \Delta_N$, and $\mu_{F}$ be any invariant semiclassical measure on $S_{E_0}N$ with support $F$ a closed set.  Let $C_{width} = \pi \tilde{\lambda}_{1}(1+3\epsilon_2)$ be the $\h$-independent factor in the spectral width of the fully localized mode $\boldsymbol{\Psi}_h$ defined in (\ref{e:qmodeN}).  Then there exist quasimodes $\tilde{\boldsymbol{\Psi}}_h$ of spectral width $\epsilon_3 \frac{h}{|\log h|}$ and central energies $E_{\h} = E_0 + \mathcal{O}\left(\frac{\h}{|\log \h|}\right)$ whose semiclassical measure on $S_{E_0}N$ contains an atom $w_{F} \cdot \mu_F$ in its ergodic decomposition where 
\bequ
w_{F} \geq \frac{\epsilon_3}{C_{width}} \frac{2}{3 \sqrt{3}} + \cO\left(\left(\frac{\epsilon_3}{C_{width}}\right)^2\right).
\eequ
\end{lem}

The proof of statement (3) of Theorem \ref{t:main} then follows by replacing the tangential eigenfunction $\phi_h$  (\ref{e:qmodeN}) with a log-scale quasimode and applying Lemma \ref{l:ansatz-separates}.

\bibliographystyle{amsplain}
\bibliography{scarring}

\end{document}

%% file: mydefs.tex
\usepackage{mathrsfs}
\usepackage{amssymb}
\usepackage{dsfont}
\usepackage{verbatim}
\usepackage{url}
\usepackage{mathtools}
\usepackage{etoolbox}

\def\N{\mathbb {N}}

\def\R{\mathbb {R}}
\def\C{\mathbb {C}}

\def\embed{\hookrightarrow}

\def\eqdef{\overset{\text{def}}{=}}
\def\defeq{\eqdef}

\DeclareMathOperator{\arcsinh}{arcsinh}

\DeclareMathOperator{\Op}{Op}
\newcommand{\Oph}{\Op_\hbar}

\newcommand\abs[1]{\left| {#1} \right|}
\newcommand\norm[1]{\left\Vert {#1} \right\Vert}

\newcommand{\wklim}[1]{\xrightarrow[#1]{\textrm{wk-*}}}

\DeclareRobustCommand
  \rddots{\mathinner{\mkern1mu\raise\p@
    \vbox{\kern7\p@\hbox{.}}\mkern2mu
    \raise4\p@\hbox{.}\mkern2mu\raise7\p@\hbox{.}\mkern1mu}}

\newcommand{\Cc}{C_{\mathrm{c}}}

\newcommand{\Ci}{C^{\infty}}
\newcommand{\Cic}{\Cc^{\infty}}

\newcommand\SL{\mathrm{SL}}

\DeclareMathAlphabet{\mathcal}{OMS}{cmsy}{m}{n}

\newcommand\cD{\mathcal{D}}

\newcommand\cF{\mathcal{F}}

\newcommand\cM{\mathcal{M}}

\newcommand\cO{\mathcal{O}}

\newcommand\cS{\mathcal{S}}

\newcommand\diff{\mathrm{d}}

\newcommand\dcross{\diff^{\kern-2pt\raisebox{1pt}{$\times$}}\kern-8pt}

\newcommand\blfootnote{\xdef\@thefnmark{}\@footnotetext}